\documentclass[a4paper,english,12pt]{amsbook}
\pdfoutput=1
\usepackage{pstricks}
\usepackage[protrusion=true,expansion=true]{microtype}
\usepackage{graphicx}
\usepackage[all]{xy}
\usepackage{hyperref}
\usepackage[utf8]{inputenc}
\usepackage{eucal}
\usepackage{mathrsfs}
\usepackage{amssymb}
\usepackage{amsxtra}
\usepackage{enumerate}
\hypersetup{colorlinks=true,citecolor=black,filecolor=black,linkcolor=blue,urlcolor=black} 
 
\newtheorem{theorem}{Theorem}[chapter]
\newtheorem{proposition}[theorem]{ Proposition} 
\newtheorem{lemma}[theorem]{ Lemma}
\newtheorem{corollary}[theorem]{Corollary}

\newtheorem{definition}[theorem]{Definition}

\theoremstyle{remark}

\newtheorem{example}[theorem]{\it Example}
 
 \def \1{\mathbb {1}}
\def \R{\mathbb {R}}
\def \RM{\mathbb {R}}
\def \NM{\mathbb{N}}
\def \ZM{\mathbb{Z}}
\def \CM{\mathbb{C}}

\def \QM{\mathbb{Q}}
\def \PM {\mathbb{P}}

\def \Im {{\rm Im\,}}

\def \Der {{\rm Der\,}}

\def \Aut {{\rm Aut\,}}

\def \p {{\rm exp\,}}
\def \Id {{\rm Id\,}}

\def \d{\partial}
 
\def\a{\alpha}
\def\b{\beta}
  
\def\g{\gamma}

\def\l{\lambda}

\def\p{\varphi}

\def\G{\Gamma}   
\def\D{\Delta}
\def \s{\sigma}

\def \to{\longrightarrow} 

\def \alg{\mathfrak{g}}

\def \< {{\langle }}
\def \> {{\rangle }}
\def \( {\left( }
\def \) {\right) }

\newcommand{\Ft}{{\mathcal F}}
\newcommand{\Gt}{{\mathcal G}}
\newcommand{\Ht}{{\mathcal H}}
\newcommand{\It}{{\mathcal I}}

\newcommand{\Mt}{{\mathcal M}}

\newcommand{\Ot}{{\mathcal O}}
\newcommand{\Pt}{{\mathcal P}}

\newcommand{\Xt}{{\mathcal X}}

\newcommand{\lra}{\longrightarrow}
\renewcommand{\mod}{{\rm  mod\,}}
\title[Th\'eorie KAM]{{\sc  Th\'eorie KAM}}
\author{ Mauricio  Garay }

\makeindex
 
\begin{document}
 \thispagestyle{empty}
 \begin{center}{\LARGE  \bf KAM THEORY\\} 
  \vskip0.8cm { \large \sc M. Garay and D. van Straten}\\
  \vskip0.5cm
  \rule{0.3\textwidth}{0.5pt}\\
    \vskip1cm
    
 { \large \sc  \textbf{III \\ \vskip0.5cm Applications}}
 
 \vskip1cm
  \begin{figure}[ht]
  \begin{minipage}[b]{0.6\linewidth}   
  \includegraphics[height=0.6\linewidth,width=0.75\linewidth]{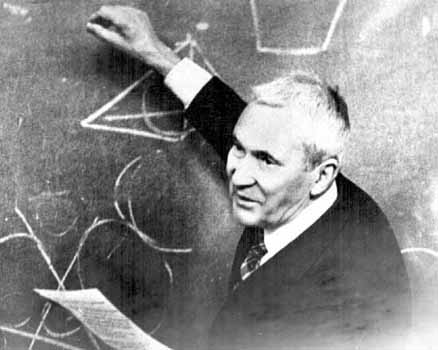}
 
  \end{minipage}
   \begin{minipage}[b]{0.60\linewidth}
     \includegraphics[height=0.42\linewidth,width=0.75\linewidth]{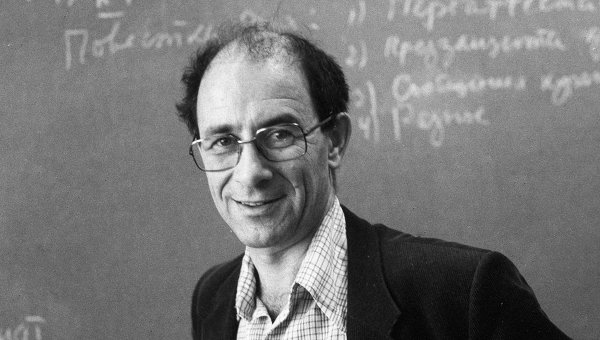}

             \end{minipage}\hfill
         \begin{minipage}[b]{0.40\linewidth}
           \includegraphics[height=0.9\linewidth,width=0.6\linewidth]{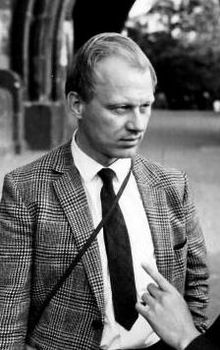}
     \end{minipage}
  \end{figure}

  \end{center}

\newpage
\ \\
\thispagestyle{empty}

\clearpage
%
%
%

\tableofcontents

 \setcounter{part}{2}
 \setcounter{chapter}{13}
 \part{Applications}
\chapter{Hypersurface singularities}
The group of diffeomorphisms $Diff(M)$ acts on the space of smooth functions $C^{\infty}(M)$ on a manifold $M$ and 
the partition of the function space into orbits is of considerable interest. 
If we fix a point $p \in M$, we can look at the local version of this problem: we consider
germs of functions at $p$, and look at the action of the sub-group of diffeomorphisms that 
fix $p$. If we introduce coordinates $x_1,x_2,\ldots,x_n$ 
for $M$ at $p$, we are dealing with germs of functions $f(x_1,x_2,\ldots,x_n)$ and the group
of coordinate transformations preserving the origin. Here we will study the related actions
on formal and convergent power series. Our presentation will be rather brief, as this material 
is covered at many places in much more detail. The main purpose is to illustrate the fact that 
the formal arguments can be lifted directly to the level of Kolmogorov spaces to give a 
corresponding result for convergent series.

\section{Formal finite determinacy\index{Finite determinacy}}
We fix a field $K$ of characteristic $0$ (most often $\RM$ or $\CM$) and consider the $K$-algebra
\[R=K[[ x_1,x_2,\dots,x_n ]]\]
of formal power series. It is a local ring with maximal ideal
\[\Mt:=\{ f \in R :f(0)=0 \}\]
consisting of series without a constant term.
The powers of the maximal ideal form a separated filtration on $R$ i.e.:
\begin{align*}
R \supset \Mt& \supset \Mt^2 \supset \ldots \supset \Mt^k \supset \cdots\\
\bigcap_k \Mt^k &=(0). \end{align*}
 The algebra $R$ is naturally equipped with the {\em $\Mt$-adic topology} for which the $\Mt^k$ form a basis of $0$-neighbourhoods\index{$\Mt$-adic topology}. This means that two series are close to each other if their
 coefficients are equal up to some high order.  The algebra $R$ is {\em complete} with respect to this topology.\\
 
Any automorphism $\p \in Aut(R)$ preserves the filtration and is determined by the $n$
power series
\[ y_i:=\p(x_i) \in \Mt,\;\;\;i=1,2,\ldots,n\]
as for any power series $f \in R$ one must have
\[ \p(f)=f(y_1,y_2,\ldots,y_n) .\]
In this sense $Aut(R)$ is the same thing as the group of formal change of variables: it acts on $R$ via coordinate transformations and the map
$$x_i \mapsto y_i(x_1,x_2,\ldots,x_n) $$
is a formal change of variables. This is a very peculiar feature of algebra automorphisms: if
$x_i$ is mapped to $y_i$ then $x_i^2$ is mapped to $y_i^2$ and so on. Therefore the knowledge of
the map on the coordinates, that is on a finite dimensional $K$-vector space, is sufficient to reconstruct the map on the whole algebra.

More generally, an $n$-tuple $(y_1(x),y_2(x),\ldots,y_n(x))$
of power series from $\Mt$ defines a endomorphism 
\[R \to R,\;\;x_i \mapsto y_i(x)\]
of the ring $R$. The {\em formal inverse function theorem} states that it is an automorphism if and only if the jacobian determinant is invertible:
\[ \det\left( \frac{\partial y_i}{\partial x_j}\right) \not\in \Mt .\]

\begin{definition}
Two series $f, g \in R$ are called {\em right equivalent} \index{right equivalence}, if they belong to the same orbit of the action of $Aut(R)$ on $R$:
\[\p(f)=g\ \textup{ or equivalently:}\;\;\;f(y_1,y_2,\ldots,y_n)=g(x_1,\dots,x_n) \]
\end{definition}

Let us look first at power series in one variable, i.e. $R=K[[x]]$.
A non-zero series $f \in R$ of order $k$ can be written as
written as
\[ f(x)=\alpha x^k+\sum_{n \ge k+1}a_nx^n=\alpha x^k(1+b_1x+b_2x^2+\ldots),\]
where $\alpha \neq 0$.
The series
\[ y(x)=x(1+b_1x+b_2x^2+\ldots)^{1/k}=x+\frac{b_1}{k} x^2+\ldots\]
defines a formal change of variables such that
$$f(y)=y^k. $$
Therefore the associated automorphism defined by
$$\p(x):=x(1+b_1x+b_2x^2+\ldots)^{1/k} $$ maps $f$ to $\alpha x^k :$
\[ \p(f)=\a x^k \]

If the field $K$ contains a $k$-th root of $\alpha$, a further coordinate
change $x \mapsto \alpha^{1/k} x$ shows that $f$ is right equivalent to the
pure monomial $x^k$. This describes the right equivalence classes in the 
one-variable case:

\begin{proposition} Let $K$ be algebraically closed, and let $R=K[[x]]$ be the
formal power series ring. Then $R$ is the countable union of the orbits of the 
monomials $x^k,\ k \in \NM$ and $\{0\}$ 
under the action of $Aut(R)$.
\end{proposition}

The module of derivations $\Theta_{R/K}=Der_K(R)$ of the ring $R$ can be identified with 
{\em formal vector fields}
$$v=\sum_{i=1}^n a_i(x)\d_i,\ a_i \in R,\;\;\; \d_i:=\frac{\d}{\d x_i}, $$
and is filtered by powers of $\Mt$ in a similar way. If we denote by $\Der_{K}(R)^{(k)}$ the space  of derivations with coefficients in $\Mt^k$, we have a filtration
\[ Der_K(R) \supset Der_K(R)^{(1)} \supset Der_K(R)^{(2)} \supset \ldots \supset 
Der_K(R)^{(k)} \supset \ldots \]

Clearly, if $f \in \Mt^k$ and $v \in \Der(R)^{(l)}$ then $v(f) \in \Mt^{k+l-1}$. 
We can try to define the exponential of  $v \in Der_K(R)$ using the
exponential series:
\[\exp(v)(f)=f+v(f)+\frac{1}{2!}v(v(f))+\ldots\]
A vector field $v \in Der_K(R)^{(2)}$ maps $\Mt^k$ to $\Mt^{k+1}$ and as
a result, the above series converges in the $\Mt$-adic topology of $R$, that
is, order by order. So we obtain a well-defined exponential mapping
$$\exp:\Der_K(R)^{(2)}\to Aut(R)$$
in the most na\"ive way, just using the power series definition of the exponential.
Clearly, the automorphism $\exp(v)$ thus defined has the property that 
\[ exp(v)(f)-f \in \Mt^2\] 
We say that it is {\em tangent to the identity} meaning that it is the identity modulo $\Mt^2$.

More generally, we can define the sub-group
\[Aut_K(R)^{(k)}:=\{\phi \in Aut_K(R)\;\;|\;\;\phi(f)-f \in \Mt^k\}\]
of automorphisms tangent to order $(k-1)$ to the identity, and 
clearly $\exp$ maps $\Der_K(R)^{(k)}$ to $Aut_K(R)^{(k)}$.

\begin{definition} We define the {\em tangent space}\index{tangent space} $Tf$ at $f$ as the image  of the map
$$\Der(R) \to R,\ v \mapsto v(f) .$$
\end{definition} 

The above map is $R$-linear and therefore $Tf$ is an $R$-module, commonly called the {\em Jacobian ideal}\index{Jacobian ideal} $J_f$, that is,
the ideal generated by the partial derivatives of $f$:
\[ Tf=J_f:=(\partial_1 f,\partial_2 f,\ldots, \partial_n f) \subset R\]

In this setting, the normal space\index{normal space}
\[ Nf:=R/J_f .\]
is often called the {\em Milnor-algebra}\index{Milnor algebra} 
of $f$.  Its dimension is called the {\em Milnor number}\index{Milnor number} and is denoted by $\mu(f)$:
\[ \mu(f):=\dim Nf .\]
For instance, if $f=x^{k+1}$, then the Jacobian ideal is $\Mt^{k}$ and the Milnor algebra is generated by the classes of $1,x,\dots,x^{k-1}$, $\mu(f)=k$.

The fact that the tangent space $Tf$ has the structure of an $R$-module is a special feature of singularity theory. 
The reason is that there are no differential relations involved in the action of the automorphism group. This fundamental property explains the success of commutative algebra to study group actions in this context. This is already reflected in the following simple lemma.

\begin{lemma} If $\mu:=\mu(f)<+\infty$, then 
$$ \Mt^{\mu} \subset Tf.$$
\end{lemma}
\begin{proof}
The ideal $\Mt^\mu$ is generated as $R$-module by the monomials $x^I$ with $|I|=\mu$. We can write $x^I$ as a product of monomials of degree one:
$$x^I=x_{i_1}\dots x_{i_\mu}.$$
As $Nf$ is of dimension $\mu$, the classes of any $\mu$ monomials together with the class of $1$  are linearly dependent. In particular there exists a relation of the form
$$\a_0+\a_1 x_{i_1}+\a_2 x_{i_1}x_{i_2}+\dots+\a_\mu  x_{i_1}\dots x_{i_\mu}=0 \in Nf$$
If $\a_k$ ist the first non-zero coefficient in this relation, we can write this
relation in the form
$$x_{i_1}\dots x_{i_k}(\a_k+g) \in J_f,\;\;\; g \in \Mt. $$
As in the local ring $R$ any element which is not in the maximal ideal is invertible, this shows that $x_{i_1}\dots x_{i_k} \in J_f$ and consequently $x^I \in J_f$. This proves our assertion.
\end{proof}

We immediately conclude:

\begin{corollary}
\label{C::muplus2}
The image of the map $\Der_K(R)^{(2)} \to R, v \mapsto v(f)$ contains
$\Mt^{\mu+2}$.
\end{corollary}

We know that in the one variable case any formal power series can be reduced to the first non-vanishing term by an automorphism. For the case of more variables, this statement can be generalised as follows:

\begin{proposition} If $\mu(f) <\infty$, then for any element of $g \in \Mt^{\mu+2}$ there exists an automorphism $\p \in \Aut_K(R)$ such that
$$\p(f+g)=f. $$
\end{proposition}
\begin{proof}
Use induction on the order of $g$. Assume that there exists an automorphism 
$\p_k$ such that: $$\p_k(f+g)=f+g_k, g_k \in \Mt^{\mu+2+k} $$
If follows from corollary~\ref{C::muplus2} that there exists a derivation $v_{k+1} \in \Der(R)^{(2)}$ such that
$$v_{k+1}(f)=g_k. $$
As $v_{k+1} \in \Der_K(R)^{(2)}$ and $g_k \in \Mt^{\mu+2+k}$ we have
\[ \exp(v_{k+1})(g_k)-g_k \in \Mt^{\mu+2+k+1},\]
so we get
$$e^{-v_{k+1}}(f+g_k)=f\ \mod\ \Mt^{\mu+2+k+1}=:f+g_{k+1} $$
If we set $\p_{k+1}=e^{-v_{k+1}}\p_k$ we then have
\[\p_{k+1}(f+g)=e^{-v_{k+1}}(f+g_k)=f+g_{k+1}\]
So we can repeat the procedure and in this way we obtain a sequence of 
derivations $v_k$ and the automorphism $\p$  given by
$$\p=\lim_{k \to +\infty} e^{-v_k}\cdots e^{-v_0} $$
maps $f+g$ to $f$.
\end{proof}

This is the classical {\em finite determinacy theorem}\index{finite determinacy} for series with finite Milnor number: terms of degree $\ge \mu+2$ can be omitted, so any series with $\mu <\infty$ is right equivalent to a polynomial; for $\mu=1$ we deduce the formal Morse lemma.

Note that for a general  $v \in Der_K(R)$ the exponential series does not define an automorphism 
of $R$. Only if $v \in Der_K(R)^{(2)}$ the series converges in the $\Mt$-adic topology to an automorphism tangent to the identity. This is the origin of the exponent  $\mu+2$ in the statement of the finite determinacy theorem.
 
However, if $v \in Der_K(R)^{(1)}$, then sometimes one can make sense of the exponential series
as a well-defined as an automorphism of $R$. This is the case if $K=\RM$ or $\CM$ for which the exponential
of a linear map is always well-defined.  We record the following variant of the above finite determinacy theorem:

\begin{proposition} ($K=\RM$ or $\CM$.) If $\mu(f) <\infty$, then there is a neighbourhood 
$U$ of the origin in $ \Mt^{\mu+1}$ such that for any $g \in U$, there exists an automorphism $\p \in \Aut_K(R)$ with
$$\p(f+g)=f .$$
\end{proposition}
For instance, consider the case $n=1, f=x^2$ then the neighbourhood may be defined by
$$U=\left\{ \sum_{i \geq 2}a_i x^i: | a_2|<1 \right\}. $$
The condition defining $U$ ensures that $f+g$ has a non-zero quadratic part.
\section{Analytic finite determinacy}
Now let us look at the ring of convergent power series
\[R:=\CM\{ x_1,\dots,x_n \} .\] 
This is also a local ring filtered by the powers of its maximal ideal
$$\Mt:=\{ f \in R :f(0)=0 \} .$$
and, as before, any automorphism $\p \in Aut(R)$ is given by an $n$-tuple of convergent power series
\[y_i:=\p(x_i) \in \Mt,\;\;\;i=1,2,\ldots,n .\]
The module of derivations $\Theta_R=Der(R)$ is the module of analytic vector fields
$$v=\sum_{i=1}^n a_i(x)\d_i,\ a_i \in R,\;\;\; \d_i:=\frac{\d}{\d x_i},$$
and we define the (analytic) normal space $Nf$ to be the cokernel of the map

\[ Der(R) \to R,\;\;\; v \mapsto v(f) .\]
So again
\[ Nf=R/J_f ,\]
but now $J_f$ is the ideal generated by the partial derivatives in the ring $R=\CM\{x_1,x_2,\ldots,x_n\}$ of
convergent power series. The (analytic) Milnor number is defined as before
\[\mu(f):=\dim Nf .\]
An identical proof shows that if $\mu(f)<\infty$, then $\Mt^{\mu} \subset J_f$ and therefore that
the image of the map
\[ Der(R)^{(2)} \to R, v \mapsto v(f)\]
contains  $ \Mt^{\mu+2}$.

\begin{proposition} If $\mu(f) <\infty$, then for any element of $g \in \Mt^{\mu+2}$ there exists an automorphism $\p \in \Aut(R)$ such that
$$\p(f+g)=f. $$
\end{proposition}
\begin{proof}
The series $f$ defines a holomorphic function on some neighborhood $V$ of the origin, $f \in \Ot(V)$. On $V$ we consider the sheaf $\Theta=\Der(\Ot)$ of holomorphic vector fields and the map of sheaves
\[  \rho:\Theta \to \Ot,\;v \mapsto v(f)\]
Cartan's Theorem $\a$ (see Appendix~\ref{T::Cartan}) implies that
\begin{enumerate}[{\rm i)}]
\item there exists a fundamental system of polycylinders $\D=(\D_\rho)$ for which $\Theta^c(\D)$, $\Ot^c(\D)$ are Kolmogorov spaces.
\item the induced Kolmogorov space morphism $\rho^c$ admits a local inverse $B$  over its image.
\end{enumerate}
The  Kolmogorov version of the map
\[  \rho^c:(\Theta^c)^{(2)} \to \Ot^c(\D),\;v \mapsto v(f)\]
has an image containing 
$$M=\Ot^c(\D)^{(\mu+1)}=\left(\Mt^{\mu+1} \right)^c(\D). $$ The preimage
of $M$ is a Kolmogorov subspace $\alg$. 

To apply the general normal form theorem  in its homogeneous form (Theorem~13.3), we must check three conditions:
\begin{enumerate}[{\rm 1)}]
\item $u(f) \in M$ for any $u \in \alg $ (true by definition of $\alg$).
\item $e^u(f+M) \subset (f+M)$ holds because $u^k(f) \subset M$ for $k>0$.
\item the map $\rho$ admits a local right quasi-inverse.
\end{enumerate}
All three condition are satisfied. The theorem implies the existence of a neighbourhood of the origin 
$$U=\Xt(r) \cap \Xt(R,k,l,\l)$$ such that for $g \in U$, there exists a sequence of derivations $(v_1,\dots,v_n,\dots)$ such that 
the composed automorphisms 
 $$e^{v_n}  \cdot e^{v_{n-1}} \cdots e^{v_{0}} $$ converge to a partial morphism $\p$ which maps $f+g$ to the restriction of $f$ to a smaller neighbourhood. 
 
If we now pass to the direct limit the partial morphism defines a ring automorphism sending $f+g$ to $f$ and the direct limit of $U$ contains $\Mt^{k+l+1}$. As we know from the formal case that for any $N \geq \mu+2$, the orbit
contains $f+g$ an element in $f+\Mt^N$ this concludes the proof.
\end{proof}
In the one dimensional Morse example, we saw that the bound $(k+l+1)$ given by  the case of convergent power series is the optimal bound $\mu+2$. 

Note that, except from this, the proof is almost identical to the proof in the formal case. Indeed, our basic strategy is to
lift all constructions from the formal level to the level of Kolmogorov spaces to obtain simple and transparent 
proofs in the convergent case.  As we shall now see, the versal deformation theorem provides another example which is from an abstract viewpoint identical to that of finite determinacy.

Nota also that the theorem holds more generally for any subspace $M$ which satisfies the  condition 1), 2) and
3).

\section{Formal versal deformations\index{versal deformations}}
\subsection*{The concept of versal deformation}
The problem of versal deformations is a variant with parameters of the previous one. 
To explain this, let us start with an example. Consider the following two deformations of 
the function $f(x)=x^3$
$$F=x^3+\l_1 x+\l_2 $$
and 
$$G=x^3+3\a^2x^2+\a.$$
The second one can be obtained from the first in the following way. As
$$ x^3+3\a^2x^2=(x+\a^2)^3-3\a^4 (x+\a^2)+(\a+2\a^6)$$
we have
\[G(x,\a)=F(x+\a^2, -3\a^4,\a+2\a^6) .\]

So, by a substitution of the parameters and a (parameter dependent) coordinate transformation, the second family may be
obtained from the first. We say that $G$ can be {\em induced} from $F$. In this example, we get the following
picture:\\
\vskip0.3cm
\begin{figure}[htb!]
\includegraphics[width=\linewidth]{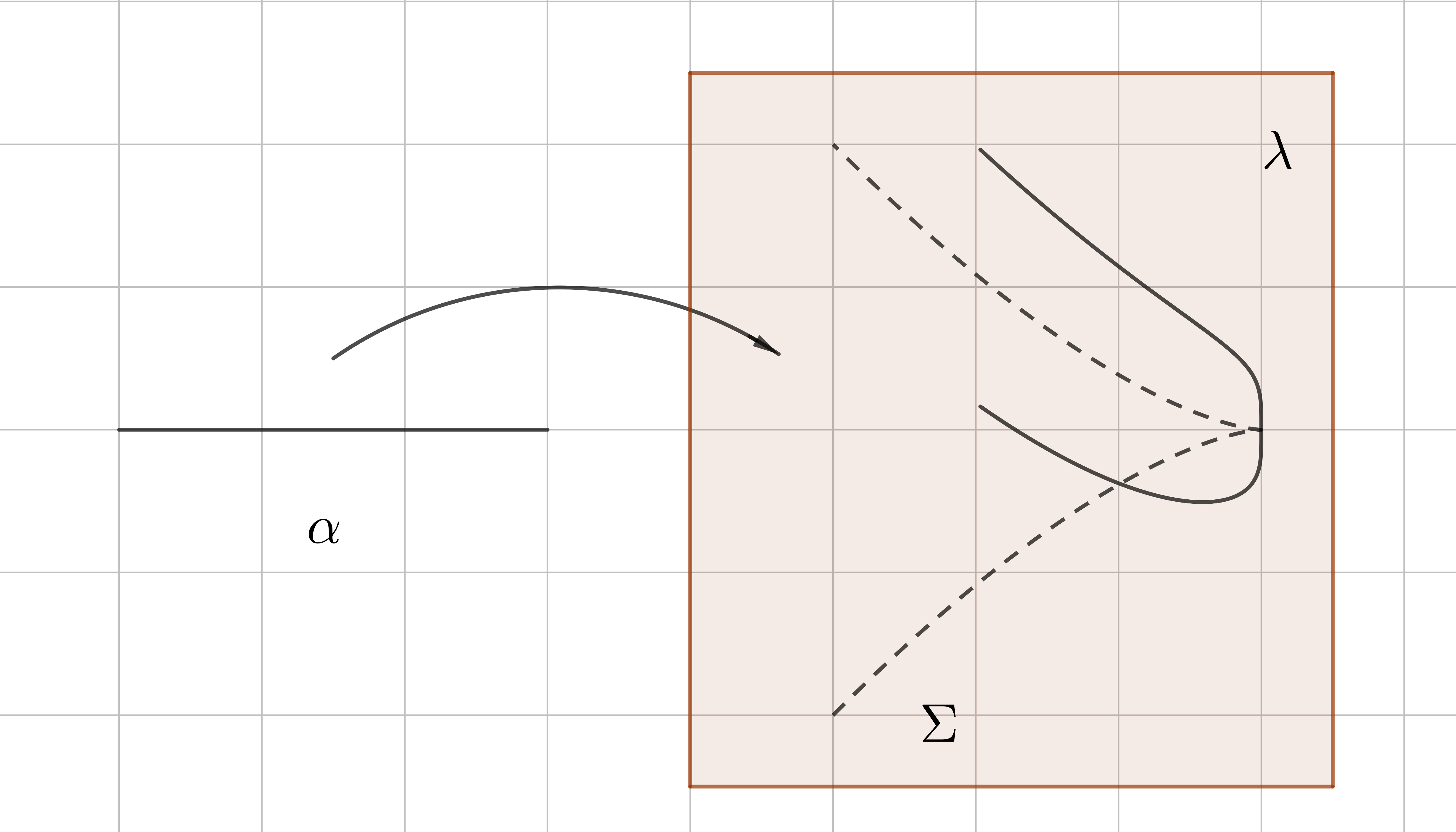}
\end{figure} 
\vskip0.3cm
The dotted curve $\Sigma$, called the {\em discriminant}, corresponds to the values of $\l$ for which the polynomial $F$ has a double root. The other curve is the image of the inducing map:
\begin{align*}
\l_1&=-3\a^4\\
\l_2&=\a+2\a^6
\end{align*}

Let us put this relationship in a more algebraic form.  We start with our ring of formal power series
$R=K[[ x ]]$ and a ring of parameters $ S =K[[ \l ]] $. An element
$$F \in R \hat \otimes S=K[[ z,\l ]]$$
is called {\em a deformation} of $f=F(-,\l=0)$ over the base $S$. A deformation $G$ over $T$ is called {\em induced} from
$F$, if there exist ring homomorphisms 
\[  \p:R \hat \otimes S \to R \hat \otimes T,\;\;\; \psi: S \to T\] 
forming a commutative diagram
$$\xymatrix{R \hat \otimes S \ar[r]^\p \ar[d] &  R \hat \otimes T \ar[d] \\
S \ar[r]^{\psi} & T
} $$
such that $\p(G)=F$. 
A deformation is called {\em versal}, if any deformation over any base $T$ can be induced from it.  

\subsection*{Versality and group actions}
The inducing map on the parameter space needs not to be an automorphism
In the previous example, we had
$$\a \mapsto (-3\a^4,\a+\a^6) $$
So even the number of variables is not the same! Therefore versality cannot be formulated directly in terms of group actions. Nevertheless one can use a trick due to Martinet. Although we shall not apply it elsewhere, it is an important feature that it can be used in other problems involving normal forms. This one of the interesting aspects of having a formal theory: every particular trick can potentially have a wide range of applications.

So let $F \in K[[\l,x]], G \in K[[ x,\a]]$ be two deformations of a series $f$. We construct the {\em Thom-Sebastiani sum} of both deformations
 $$(F \oplus G)(x,\a,\l)=F(x,\l)+G(x,\a)-f(x). $$
 It is a deformation of $f$ which is equal to $F$ (resp. $G$) when restricted to $\a=0$  (resp. $\l=0$). Moreover $F \oplus G$ might be seen as a deformation of $F$ depending on $\a$ (a deformation of a deformation). To prove that $G$ is up to automorphism induced by $F$,   we need to find an automorphism  of $K[[ x,\l,\a]]$ which maps $F \oplus G$ to $F$ and mapping $K[[\l,\a]]$ to itself and the identity on $K[[\a]]$. 
 
 Let us go back to our previous example. In this case
 $$F \oplus G=x^3+\l_1 x+\l_2+3\a^2x^2+\a$$
The automorphism
$$(x,\l,\a) \mapsto (x+\a^2,\l_1 -3\a^4,\l_2+\a+2\a^6-\l_1\a^2,\a) $$
maps $F$ to $F \oplus G$ (note that the $\a$ parameter is  unchanged).
The inverse automorphism $\p$ sends $F \oplus G$ to $F$. It is easily computed
$$\p:(x,\l,\a) \mapsto (x-\a^2,\l_1 +3\a^4,\l_2-\a+\a^6+\l_1\a^2,\a) .$$
As $G$ is the restriction of $F\oplus G$ to $\l=0$, it is mapped, via $\p$, the restriction of
$F$ to $\p(\l)=0$. So we get the system of equations
\begin{align*}
\l_1 +3\a^4&=0\\
\l_2-\a+\a^6+\l_1\a^2&=0
\end{align*}
which as expected gives back the inducing map
\begin{align*}
\l_1&=-3\a^4\\
\l_2&=\a+2\a^6
\end{align*}

So we see that the theory of versal deformations can, in this case, be induced from that of group actions~\footnote{for further applications, it could be interesting nevertheless to formalise directly the concept of versality in terms of Kolmogorov spaces.}. 
\subsection*{Some basic local algebra}
We will need some basic facts from local algebra and start by recalling the
 following version of {\em Nakayama's lemma}:

\begin{proposition} 
\label{P::nakayama}
Consider a finitely generated module $M$ over a local ring $R$ with maximal ideal $\Mt$. Then:
elements $m_1,m_2,\ldots, m_r  \in M$ generate
$M$ if and only if the classes $\bar m_1,\bar m_2,\ldots, \bar m_n$ generate the $k:=R/\Mt$-vector space $M/\Mt M$.
\end{proposition}
\begin{proof}
The implication $\implies$ is trivial. Conversely, without loss of generality we may assume that
$$M/\Mt M=\{ 0\}$$
If this were not the case, replace $M$ by $M'=M/N$ where $N$ is any submodule with 
$$\dim N/\Mt N=\dim M/\Mt M.$$

So we now assume that $M/\Mt M=\{ 0\}$. The Nakayama lemma becomes
$$M/\Mt M=\{ 0 \} \implies M=\{ 0 \} $$
Consider any set of generators $m_1,m_2,\ldots,m_n$ for $M$. As the classes of the $m_i$ are $0$
mod $\Mt M$, we can find elements $a_{ij} \in \Mt$ such that
$$m_i =\sum_{j \geq 0} a_{ij}m_j\ a_{ij} \in \Mt M$$
Writing  out this as a matrix relation 
$$m=(m_1,\dots,m_n) \in M,\ A=(a_{ij}) \in M(n,\Mt)$$
we get that
$$(\Id-A)m=0 .$$  
As $\det(I+A)=1+r, r\in \Mt$ is a unit, it follows from Cramers rule that the matrix $I+A$ is 
invertible with inverse 
$$B:= \det(I+A)^{-1} (I+A)^{ad}$$ so we have $m=0$. 
\end{proof}

From the Nakayama lemma, one sees that the minimal number of generators of a finitely generated module $M$ is just
$\dim_k(M/\Mt M)$.

The {\em generalised preparation theorem} is an important result which implies finite generation of 
modules in a special situation. To formulate it, we consider two power series rings 
\[R=K[[z_1,\dots,z_{n} ]],\;\;\;S:=K[[\l_1,\l_2,\ldots,\l_m]]\]
and a homomorphisms of $K$-algebras
\[\p: S \to R,\;\;y_i \mapsto \p(\l_i)=\p_i(z_1,\ldots,z_n) \]
Any $R$-module $M$ then can be considered as an $S$-module via the homomorphism $\p$.
The ring $S$ plays the role of the parameter space.
\begin{theorem} 
\label{T::formal_preparation}
Let $M$ be a finitely generated $R$-module. 
Then $M$ is finitely generated as $S$-module provided that:
\[ \dim M/\Mt_S M <\infty .\]
(Here $\Mt_S=(\l_1,\ldots,\l_m)$ is the maximal ideal of $S$).
\end{theorem}
\begin{example}
Take $R=K[[x,\l]]$, $S=K[[\l]]$ and 
$$\p:S \to R $$
the inclusion. The module $M=R/(x^2+\l)$ is a finitely generated $R$-module generated by the class of $1$. It is also a free $S$-module of rank $2$ generated by the classes of $1$ and $x$.
\end{example}
\begin{proof}
Without loss of generality, we may assume that
\begin{enumerate}
\item $S=K[[\l_1,\dots,\l_m]]$ is a subring of $R=K[[x_1,\dots,x_n,\l_1,\dots,\l_m]]$.
\item $R=K[[x,\l_1,\dots,\l_m]]$ that is $n=1$.
\end{enumerate}
\noindent (1) Write $T:=K[[x,\l]]$, the homomorphism $\p:S \to R$
factors as
\[ S \hookrightarrow T \twoheadrightarrow R \]
where the first map is the inclusion and the second is
 $$x_i \mapsto x_i,\;\;\; \l_i \mapsto \p(\l_i).$$
As any set of generators for $M$ as $R$-module are also generators for $M$ as $T$-module, this proves the assertion.

\noindent (2) We use a descending induction. Assume the theorem is proved only for $n=1$. Define 
$$S_i=K[[x_1,\dots,x_i,\l_1,\l_2,\ldots,\l_m]].$$
Applying successively the result we get that $M$ is a finitely generated $S_{i}$-module for $i=n-1$, $i=n-2,\dots, i=0$.

Remark that unlike the case of the Nakayama lemma we cannot assume that $M=\{ 0 \}$ because by taking the quotient of $M$ with an $S$-module we usually loose the $R$-module structure.

Let $m_1,m_2,\ldots,m_r$ be generators of $M$ as $R$-module.
\begin{lemma} There exists $N \geq 0$ such that 
$$x^Nm_i \equiv 0\ \mod \Mt_S$$
for any $i \in \{ 1,\dots,r \}$.
\end{lemma}
\begin{proof}
 Put $d=\dim M/ \Mt_S M$, the classes
$$m_i,xm_i,\dots,x^{d}m_i $$
are not linearly independent, that is:
$$\sum_{i \geq k_i} \a_i x^i m_i \equiv 0 \ \mod \Mt_S, $$
where $\a_{k_i} \neq 0$ is the first non-vanishing coefficient.
So we have
$$x^{k_i}u_i(x)m_i \equiv 0 \ \mod \Mt_. $$
where $u_i=\sum_{i \geq k_i} \a_i x^{i-k_i}$ is a unit. Therefore
$$x^{k_i}m_i \equiv 0 \ \mod \Mt_S $$
 Defining $ N=\max (k_i)$, we get   that:
\[ x^N m_i  \equiv 0 \ \mod \Mt_S. \]
This proves the lemma.
\end{proof}
Writing out the relation of the lemma in matrix terms, we find a matrix $A$ with entries $a_{ij} \in \Mt_S$ such that
\[x^N m=A m\]
So the element $g(x,y_1,y_2,\ldots,y_n):=\det(x^NI-A) \in R$ has the property that:
$$g\cdot m_i=0,\ g(x,0)=x^N$$ for all $i=1,2,\ldots,r$ that is:
$$g\cdot M=0.$$
The formal case of the Weierstra{\ss} preparation theorem 
can be used to write
\[g=u \cdot h,\;\;\;h=x^N+a_1x^{N-1}+\dots+a_N \]
where $a_i \in S, a_i(0)=0$  and $u$ is a unit in $R$.  As $u$ is a unit, we also have 
$$h \cdot M=0$$ 
Thus $M$ can be considered as a module over the factor ring $R/(h)$. But  $R/(h)$ is a free $S$-module with basis $1,x,x^2,\ldots,x^{N-1}$, and therefore the elements
\[ x^i m_j,\;\;\;i=0,1,\ldots,N-1,\;j=1,2,\ldots,r \]
form a set of generators of $M$ as an $S$-module. This proves the theorem.
\end{proof}
\subsection*{The formal versal deformation theorem}
 \begin{theorem} A deformation $F \in R[[\l ]]$ of $f \in R$ is versal if and only if
 the classes of the $\d_{\l_i}F(-,\l=0) $'s generate the $K$-vector space $Tf=R/J_f$.
 \end{theorem}
\begin{example}  
The deformation
$$F(x,\l)=x^3+\l_1 x+\l_2 $$ is a versal deformation of $x^3$.
 Indeed:
 $$\d_{\l_1}F(-,\l=0)=x,\ \d_{\l_2}F(-,\l=0)=1,\ J_f=\Mt^2,\ K[[ x ]]/J_f=K \bar 1 \oplus K \bar x .  $$
 \end{example}
\begin{proof}
We use Martinet's trick and show that the {\em Thom-Sebastiani sum} of both deformations
 $$(F \oplus G)(x,\a,\l)=F(x,\l)+G(x,\a)-f(x) $$
 is isomorphic to $F$.  
 
 We use the notations
 \begin{align*}
 \Der_{\a,\l}K[[x,\a,\l]] :=&\Der_{K[[ \a,\l]]}K[[x,\a,\l]],\\
 \Der_{\a}K[[\a,\l]]:=&\Der_{K[[\a]]}K[[\a,\l]]
 \end{align*}
 and so on. So for example, an element of $\Der_{\a,\l}K[[x,\a,\l]]$ are formal vector of the form
\[ v=\sum_{i=1}^n A_i \frac{\partial}{\partial x_i},\;\;\;A_i \in K[[x,\a,\l]] .\]
 
We assert that the map 
 \begin{align*} \rho:&\Der_{\a,\l}K[[z,\l,\a]] \oplus \Der_{\a}K[[\l,\a]] \to  K[[z,\l,\a ]]\\
&  v \mapsto v(F \oplus G)
  \end{align*}
 is surjective. 
To show it, consider the $K[[x,\a,\l]]$ module $M:=Coker(\rho)$ and consider it as $K[[\a,\l]]$-module.
Clearly, by putting $\a=0$, we find $M':=M/(\a)M=Coker(\rho')$, where
$$\rho':\Der_{ \l}K[[z,\l]] \oplus \Der_{K}K[[\l]] \to  K[[z,\l ]], v \mapsto v(F)$$
is obtained from $\rho$ by setting $\a=0$.
If we now also divide out the $\l$, we find that $M/(\a,\l)=M'/(\l)M'$ 
can be identified with $Tf/(\d_{\l_i} F_{|\l=0})$.  But assumption, this is zero. 
If follows from Theorem~\ref{T::formal_preparation} that $M$ is finitely generated as $K[[\a,\l]]$-module
and hence is zero by Nakayama's lemma~\ref{P::nakayama}.
 
Let us denote by $\Mt_\a \subset  K[[z,\l,\a ]] $ the module  of series vanishing at $\a=0$ and filter 
the ring $ K[[z,\l,\a ]] $ by powers of $\Mt_\a$. Assuming the existence of an automorphism $\p$ over the base of the deformation such that:
 $$\p(F+G)=F\ \mod K[[z,\l,\a ]]^{(k)}=F+R \ \mod \Mt_\a^{k+1} $$
 We choose a derivation $v  \in (\Der_{\a,\l}K[[z,\l,\a]] \oplus \Der_{\a}K[[\l,\a]])^{(2)} $ such that 
 $$\rho(v)=R.$$ 
The automorphism $\p'=e^{-v}\p$ gives
$$\p'(F+G)=F\ \mod \ \Mt_\a^{k+1} $$
In this way, we construct order by order an automorphism which maps $F+G$ to $F$. This concludes the proof 
of the theorem.
\end{proof}
\section{Analytic versal deformations}
We would like to prove the existence of a versal deformation for convergent power series. So we consider the ring  
\[\CM\{ z,\l \}:=\CM\{ z_1,\dots,z_n,\l_1,\l_2,\ldots,\l_k \}\]
considered as algebra over
\[\CM\{ \l \}:=\CM\{ \l_1,\dots,\l_k\} .\] 
The ring $\CM\{ z,\l \}$ is filtrated by powers of the maximal ideal $\Mt_l$ of the ring $\CM\{\l\}$ 
consising  of the functions vanishing at $\l=0$.

\begin{theorem} A deformation $F \in \CM\{z,\l \}$ of $f \in \CM\{ z \}$ is versal if 
and only if the classes of the $\d_{\l_i}F(\l=0,-) $'s generate the vector space $Tf$.
\end{theorem}   

\begin{proof}
 We proceed like in the formal case. Given a deformation $G$ of $f$ over $\CM\{\a_1,\dots,\a_l \}$, we construct the Thom-Sebastiani sum as before $F \oplus G$.

The generalised preparation theorem (Theorem \ref{T::formal_preparation}) also holds for convergent power series,
with a near identical proof where one replaces the formal Weierstra{\ss} division theorem by its analytic
version~(see Appendix). Therefore like in the formal case, the cokernel of the map  
 \begin{align*} \rho:&\Der_{\a,\l}\CM\{z,\l,\a\} \oplus \Der_{\a}\CM\{\l,\a\} \to  \CM\{z,\l,\a \}\\
&  v \mapsto v(F \oplus G)
  \end{align*}
reduces to zero.
 
We lift the discussion at the level of Kolmogorov spaces. We choose polycylinders $\D=(\D_s)$ adapted to the
map $\rho$. Denote by $\Mt_S $ (resp. $\Theta$), the sheaf of holomorphic function (resp. derivations)   which vanish at $\a=0,\l=0$.

Like for finite determinacy, we apply the normal form theorem in its homogeneous form with 
$$E=\left(\Ot^c(\D)\right),\ M=\Mt_S^c(\D),\ \alg =\Theta^c(\D) ,\ a=F.$$

To apply the general normal form theorem  in its homogeneous form (Theorem
8.3), we must check three conditions:
\begin{enumerate}[{\rm 1)}]
\item $u(f) \in M$ for any $u \in \alg $ (true by definition of $\alg$).
\item $e^u(f+M) \subset (f+M)$ holds trivially.
\item the map $\rho$ admits a singular local right inverse over $M$ (true because of Cartan's theorem $\a$ ).
\end{enumerate}
The theorem implies the existence of a neighbourhood of the origin $U$ such that for $G \in U$, there exists a sequence of derivations $(v_1,\dots,v_n,\dots)$ such that 
the composed automorphisms 
 $$e^{v_n}  \cdot e^{v_{n-1}} \cdots e^{v_{0}} $$ converge to a partial morphism $\p$ which maps $F \oplus G$ to the restriction of $F$ to a smaller neighbourhood. 
 
If we now pass to the direct limit the partial morphism defines a ring automorphism sending any deformation of $F$ to $F$ and the direct limit of $U$ contains $\Mt_S^{k}$ for some $k$.

We know from the formal case that $F \oplus G$ contains an element in $F+\Mt_S^k$.  In this way, we get that $F \oplus G$ lies in the orbit of $F $. This proves the versal deformation theorem. 
 \end{proof}
 \section{Bibliographical notes}
 The finite determinacy theorem is proved and stated in:\\
 
{\sc J.N. Mather,} {\em Stability of $ C^\infty $ mappings, III. Finitely determined map-germs.} Publications Math\'ematiques de l'IHES, 35, 127-156, 1968.\\

The existence of versal deformations is proved in:\\ 
 
{\sc G.N. Tyurina,} {\em Locally semi-universal plane  deformations of isolated singularities in complex spaces}, 
Math. USSR, Izv, 32:3, 967-999, 1968.\\

Arnold proposed a KAM theoretic approach to Hypersurface Singularities in:\\
{\sc V.I. Arnold}, {\em A note on the Weierstra{\ss} preparation theorem.} Functional Analysis and its Applications,
1967, 1:3, 1-8. 32.\\

 The following books are classical references:\\

{\sc  J. Martinet},{\em Singularities of smooth functions and maps}, Cambridge University Press, Vol. 58, 1982.\\

{\sc V.I. Arnold, V. Vassiliev, V. Goryunov and O. Lyashko }, {\em Dynamical Systems VI: Singularity Theory I, volume 6 of Encyclopaedia of Mathematical Sciences}, Springer 1993.\\

The idea of using Thom-Sebastiani sums to prove the versal deformation theorem is due to Martinet.

 The abstract form of the preparation theorem was given by Houzel following an idea of Serre in:\\

{\sc C. Houzel}, {\em G\'eom\'etrie analytique locale I,} S\'eminaire Henri Cartan, 13 no. 2 (1960-1961), Expos\'e No. 18, 12 p. (available on Numdam).\\

There is also a functional analytic  proof, based on Riesz theory, which generalises the theorem to the case of non-linear operators:\\

{\sc C. Houzel}, {\em Espaces analytiques relatifs et th\'eor\`eme de finitude}, Mathematische Annalen, 205, 1973, 13-54.\\

{\sc M. Garay}, {\em Finiteness and constructibility in local analytic geometry}, L'Enseignement Math\'ematique, 55, 2009, pp. 3--31.

%
 
\chapter{Normal forms of vector fields}
The group $Diff(M)$ of diffeomorphisms of a manifold $M$ acts naturally on the
spaces of tensor fields on $M$. Apart from the action on the space of functions
$C^{\infty}(M)$ that we discussed in the previous section, the action on the space $\Theta(M)$
of $C^{\infty}$ vector fields is of particular interest. As a vector field
can be seen as an infinitesimal diffeomorphism, the infinite dimensional Lie-algebra of vector fields can be
seen as the tangent space to $Diff(M)$ at the identity, and the action of $Diff(M)$ on $\Theta(M)$
is an infinite dimensional version of the adjoint representation of a Lie group.
If we pick a point $p \in M$, one can look at action of the diffeomeorphisms fixing $p$ on the
germs of vector fields at $p$. Introducing coordinates $x_1,x_2,\ldots,x_n$ we may represent a vector field as
\[ v=\sum_{i=1}^n a_i \frac{\partial}{\partial x_i},\]
and where the coeffients $a_i=a_i(x)$ are smooth functions of the coordinates.
In this chapter we will study the cases where the coefficients $a_i$ are formal of convergent power series. 
Contrary to the case of hypersurface singularities, in the case of vector fields non-trivial convergence
issues play a role and provide a first non-trivial application of the normal form theorem.
\section{Vector fields}
The diffeomorphism group $Diff(M)$ of a manifold $M$ acts on itself by conjugation:
$$(\p,\psi) \mapsto \p \circ \psi \circ \p^{-1}. $$
In this formula $\p$ is viewed as a change of variables and 
$$\psi:M \to M $$
as a map that we can possibly iterate to define a discrete dynamical system on $M$.

Due to the fact that diffeomorphisms are read from right to left and automorphisms from left to right, if we regard diffeomorphisms as automorphisms of $C^\infty(M,\RM)$, then the action is the other
way
$$ (\p,\psi) \mapsto \p^{-1} \psi \p.$$
So we have two languages, geometric and algebraic, for the same object. 

Integrating a vector field $v$ at time $t$  yields a diffeomorphism $\p_t$. A change of variables transforms
$\p_t$ and also the vector field $v$ into a vector field $\p \cdot v$.
This is the adjoint representation of the diffeomorphism group $G=Diff(M)$ into the Lie algebra of vector fields $\alg=\Theta(M)$.  

In local coordinates $x=(x_1,x_2,\ldots,x_n)$, the vector field takes the form
\[ v=\sum_{i=1}^n a_i \frac{\partial}{\partial x_i} .\]
If we transform $v$ to a different coordinate system $y=(y_1,y_2,\ldots,y_n)$, 
one has
\[\frac{\partial}{\partial x_i}=\sum_j \frac{\partial y_j}{\partial x_i} \frac{\partial}{\partial y_j},\]
and so one has
\[v=\sum_{i,j} a_i(x)\frac{\partial y_j}{\partial x_i} \frac{\partial}{\partial y_j} \] 
This formula   describes the {\em effect of a diffeomeorphism} 
\[ \p: (x_1,x_2,\ldots,x_n) \mapsto (y_1(x),\ldots,y_n(x))\]
on the vector field $v$, in which case we would write   
\[\p \cdot v=\sum_{i,j} a_i(x)\frac{\partial y_j}{\partial x_i} \frac{\partial}{\partial y_j} .\] 

If the vector field $v$ is now viewed as a derivation of the ring $R=C^\infty(M,\RM)$ 
then the action of an automorphism $\p: R\to R$ on $v$ is by conjugation.   This is the adjoint representation
of the group $G=Aut(R)$ on the module of derivations. Therefore these two procedures produce the same result:
 $$\p \cdot v=\p^{-1} v \p. $$

Concretly this means that computing into two different ways leads to the same result. Let us consider a simple example of a vector field defined on the one-dimensional line
described with a coordinate $x$:
\[ v=a(x) \frac{\partial}{\partial x} .\]
Now consider the local diffeomorphism 
$$\p: x \mapsto 2x+x^2=:y $$
with inverse
$$ \p^{-1}: y \mapsto -1+\sqrt{1+y}=\frac{1}{2}y-\frac{1}{8}y^2+\ldots $$
We have
$$\frac{\partial y}{\partial x}=2+2x=2\sqrt{1+y} , $$
so the vector field $v$ is transformed into
$$\p \cdot  v=2\sqrt{1+y}\ a(-1+\sqrt{1+y})\d_y. $$

If we now localise at the origin, so that now $v$ is a derivation of the local
ring of $C^\infty$-germs: 
$$R \to R,\ f(x) \mapsto a(x)f'(x) $$
and $\p$ is the automorphism
$$\p:R \to R,\ x \mapsto 2x+x^2 $$
and
$$\p^{-1}:R \to R,\ x \mapsto -1+\sqrt{1+x} $$
Composing with $\p$, we get that:
\begin{align*}
\p^{-1}a(x)\d_x \p(x)&=\p^{-1} \left(a(x)\d_x(2x+x^2)\right)\\
 &=\p^{-1}(a(x))\p^{-1}(2+2x)\\
 &=\p^{-1}(a(x))(2+2\p^{-1}(x))\\
 &=a(-1+\sqrt{1+x})2\sqrt{1+x} . 
 \end{align*}
which is of course the same result if we replace the letter $x$ by the letter~$y$. We see in this simple example that, contrary to the case of singularity theory, the action is quite involved already in the one dimensional case.  
 
\section{Adjoint orbits}
Having clarified the relation between change of variables in vector fields 
and adjoint action on derivations, we may now consider the classical theory 
of formal normal forms for vector fields. This theory goes back to the work
of Poincar\'e and thus predates the simpler normal form theory for hypersurface 
singularities.

So let $K$ be a field of characteristic zero and consider the ring $R:= K[[x_1,\dots,x_n]]$ of formal power series in $n$
variables.  A formal vector fields is an expression of the form:
 $$ \sum_{i \geq 0}a_i(x)\d_{i},\ \d_i:=\frac{\d}{\d x_i}$$
 with $a_i \in R$. These are the derivations of the ring $R$. We investigate the orbit of a derivation under the automorphism group of $R$ or which is the same the action of changing variables in a vector field. If we regard a vector field $v$ as derivation of $R$, the action of an automorphism
 $\p \in \Aut(R)$ is given by:
 $$ w=\p v\p^{-1}.$$
 Now if we take a vector field $u \in \Der(R)^{(2)}$, it exponentiates and gives rise to an automorphism $e^u$. The action of such an automorphism
 is given by
 $$e^u(f)=f+u(f)+\frac{1}{2!}u \circ u(f)+\cdots $$
 It acts on a derivation $w$ as:
$$e^v w e^{-v}= w+[v,w]+\dots$$
where the dots stand for terms at least quadratic in $v$. So we recover the fact that the infinitesimal is given by the Lie bracket: if
\[v=\sum_i a_i \partial_i, \;\;w=\sum_i b_i \partial i,\]
then 
\[ [v,w]=\sum_i c_i \partial_i,\]
where
\[ c_i =\sum_j a_j \partial_j b_i -b_j\partial_j a_i .\]

Like for functions, our aim is to study orbits, called {\em adjoint orbits}, under this action. It can be considered as infinite dimensional version of the example studied in section 5.2. Let us start the elementary case of the rectification of
a non-singular vector field 
\[ v \in Der(R)^{(0)}\setminus Der(R)^{(1)} .\]

\begin{proposition} Assume that $v =\sum_{i \geq 0}a_i(x)\d_{i}$  is such that one of the $a_i$'s does not vanish at $0$, then the derivation $v$ lies in the adjoint orbit of $\d_1$.
\end{proposition}
\begin{proof}
The automorphism group acts transitively on linear derivations thus we may assume that:
$$v=\d_1 \ \mod \Der(R)^{(1)}. $$
Let us show by induction on $k$ that the orbit of $v$ contains an element of the form
$$v'=\d_1 \ \mod \Der(R)^{(k)} .$$
Write
$$v'=\d_1+w,\ w \in \Der(R)^{(k+1)} $$
If $\xi=\sum \xi_i\partial_i$ is a vector field, then
\[[\xi,\partial_1]=-\sum_i \partial_1 \xi_i \partial_i \]
So we can always find $\xi \in \Der(R)^{(k+2)}$ such that
$$[\xi,\d_1]+w = 0 \ \mod \ \Der(R)^{(k+2)}. $$
We have
$$e^{\xi} v' e^{-\xi}=\d_1\ \mod \Der(R)^{(k+2)}. $$
This proves the proposition.
\end{proof}

\section{Formal Poincar\'e-Dulac theorem}

Now we turn to the case of a linear vector field
$$v=\sum_{i=1}^n \l_i x_i \d_i. $$
and study its perturbations
\[ v+w, w \in Der(R)^{(2)}\]
obtained by adding higer order terms.
The infinitesimal action is given by a commutator and an explicit computation shows that:
\begin{align*}
[x^I\d_i,v]&=\sum_{j=1}^n \l_i[x^I\d_i,x_j\d_j]\\
 &=\sum_{j=1}^n (\l_ix^I\d_i x_j)\d_j+\sum_{j=1}^n \l_j x_j[x^I\d_i,\d_j]\\
 &=\l_i x^I \d_i-\left(\sum_{j=1}^n \l_jI_j\right) x^I\d_i=(\l_i-(\l,I))x^I\d_i
 \end{align*}
 where $(-,-)$ denotes the Euclidean scalar product. We have proved the

\begin{proposition}  Let $v \in \Der(R)$ be a derivation of the form
$$v=\sum_{i=1}^n \l_i x_i \d_i. $$
The linear map
 $$ \Der(R) \to   \Der(R), \xi \mapsto  [\xi,v] $$
is diagonal in the basis $x^I\d_i$ with eigenvalues $(\l_i-(\l,I))$.
\end{proposition}

The situation is therefore similar to that encountered when we discussed normal forms of Hamiltonian functions. 
Remark that we always have
\[ [x_i\partial_i, v]=0\]
but the these linear vector fields $x_i \partial_i$ exponentiate to scalings of
the coordinates.

We call $v$ {\em non-resonant}\index{non-resonant vector field}  if the commutant 
reduces to the linear span of the $x_i\partial x_i, i=1,2.\ldots,n$:
$$[\xi,v]=0 \iff \xi \in \bigoplus_{i=1}^n \CM x_i\d_{x_i}. $$ 
Otherwise, $v$
is called {\em resonant}, and a monomial $x^I\partial_i$ commuting with $v$ a {\em resonant monomial}.
According to the previous computation:
$$ x^I\d_i \text{ resonant} \iff (\l_i-(\l,I))=0.$$
For instance, in the two variable case $x_1\d_x+\a x_2\d_2$ is resonant precisely when $\a \in \QM$.

It is natural to study the transformations under the sub-group of automorphisms tangent to the identity:
\[ Aut^{(2)}(R):=\{ \p \in \Aut(R)\;|\;\p=Id  \ \mod \ \Mt^2\} .\] 

\begin{theorem} Let $v \in \Der(R)$ be a derivation of the form
$$v=\sum_{i=1}^n \l_i x_i \d_i. $$
The commutant of $v$:
$$F:=\{ w \in \Der(R)^{(2)}:[v,w]=0 \}. $$ 
is a transversal to the adjoint action of $Aut^{(2)}(R)$ on $v+\Der(R)^{(2)}$. 
\end{theorem}
In other words, for any $w \in \Der(R)^{(2)} $,
there exists an automorphism $\p \in Aut^{(2)}(R)$
such that $$\p(v+w)\p^{-1} =v+r,\ r \in   \Der(R)^{(2)}, $$
and $[r,v]=0$.

\begin{example} Take for $K$ any field containing $\QM(\sqrt{2})$ (for instance $\RM$ or $\CM$) and consider the vector field
$$v=\sqrt{2}x\d_x+y\d_y .$$
According to our previous computation
\begin{align*}
[x^iy^j\d_x,v] &=(\sqrt{2}-\sqrt{2}i-j)x^iy^j\d_x \\
[x^iy^j\d_y,v] &=(1-\sqrt{2}i-j)x^iy^j\d_y
\end{align*}
As $1$ and $\sqrt{2}$ are $\QM$-independent, the vector field is non resonant and the commutator reduces to the two dimensional vector space
$$K(x \d_x)\oplus K(y\d_y) \subset \Der(R). $$
The Poincar\'e-Dulac theorem states that any formal vector field of the form
$v+w $ with $w \in \Der(R)^{(2)}$ belongs to the adjoint orbit of $v$.
\end{example}
\begin{example}
Consider the vector field
$$v=2x\d_x+y\d_y .$$
According to our previous computation we have:
\begin{align*}
[x^iy^j\d_x,v] &=(2-2i-j)x^iy^j\d_x ,\\
[x^iy^j\d_y,v] &=(1-2i-j)x^iy^j\d_y .
\end{align*}
The vector field is now resonant since:
$$[y^2\d_x,v]=0. $$
The commutator of $v$  is a three dimensional vector space:
$$\{ w \in \Der(R):[v,w]=0 \}=K\, x \d_x \oplus K\, y \d_y \oplus K\, y^2\d_x. $$
Among these vectors, only $y^2\d_x$ lies in $ \Der(R)^{(2)}$. The Poincar\'e-Dulac theorem states that any formal vector field of the form
$v+w $ with $w \in \Der(R)^{(2)}$ belongs to the adjoint orbit of 
$$(2x+\a y^2)\d_x+y\d_y$$
for some  $\a \in K$. So we get a polynomial normal form depending on one parameter.
\end{example}
\begin{example}
The situation is radically different for the vector field:
$$v=x\d_x-y\d_y .$$
According to our previous computation 
\begin{align*}
[x^iy^j\d_x,v] &=(1-i+j)x^iy^j\d_x ,\\
[x^iy^j\d_y,v] &=(-1-i+j)x^iy^j\d_y .
\end{align*}
the commutator is now generated by monomials of the form
$$ x^{j+1}y^{j}\d_x,\ x^{i}y^{i+1}\d_y.$$
This means that our normal form is not a polynomial but a formal power series of the form
$$v+\sum_{i \geq 1} \a_i x^{i+1}y^{i}\d_x+ \sum_{i \geq 1} \b_i x^iy^{i+1}\d_y.$$
\end{example}
We now prove the theorem.
\begin{proof}

We use induction on the order. Assume that we have some derivation
$$v+r_k+w', w' \in \Der(R)^{(k+1)},\ k \geq 1, $$
with $[r_k,v]=0$ in the orbit of $v+w$.  There exists $\xi_k \in \Der(R)^{(k+1)}$, $s_k \in F$ such that
$$[\xi_k,v]+w'=s_k\ \mod  \Der(R)^{(k+2)}  .$$
We have
$$e^{-\xi_k}v e^{\xi_k}=v+r_k+s_k\  \mod  \Der(R)^{(k+2)} $$
and $r_{k+1}=r_k+s_k \in F$.
As we are only using exponentials of elements from $Der(R)^{(2)}$, the
resulting automorphism acts trivially on $\Mt/\Mt^2$.
\end{proof}

A vector field of the form 
$$v+r,\ [r,v]=0 , r \in  \Der(R)^{(2)}$$ 
is said to be in  {\em Poincar\'e-Dulac normal form}\index{Poincar\'e-Dulac normal form}. There is an unique element of this form in the orbit under $Aut^{(2)}(R)$. 
This is again due to the fact that for a diagonal operator the kernel is transversal to the image.

It appears that there are two essentially distinct cases to consider.\\

{\em Poincar\'e domain}:  $0$ lies {\em not} in the convex hull of the $\l_i$'s 
In this case, the commutant is a finite dimensional vector space.
 
{\em Siegel domain:} $0$ lies in the convex hull of the $\l_i$.\\  
In this case the commutant is either infinite dimensional or trivial.\\

For instance, going back to the previous example the vector $(2,1) \in \CM^2 $ belong to the Poincar\'e domain while  $(1,-1)$ lies in the Siegel domain.
\section{The Poincar\'e-Siegel theorem}
We now lift the discussion at the level of Kolmogorov space so now $K=\CM$ and $R$ is the ring $\CM\{ x_1,\dots,x_n \}$ of
convergent power series. As usual we consider the family $D=(D_s)$ of polydiscs centred at the origin of radius $s$. 
The normal form theorem (Theorem 13.3) implies the following {\em finite determinacy theorem for vector fields}

\begin{theorem} Consider a vector field $v \in E:=\Der(\Ot^c(D))$
and assume that for $k > 1 $  the map 
$$\rho:\Der(E)^{(2)} \to \Der(E)^{(k)},\ w \mapsto [v,w] $$ admits a local right-inverse.
Then there is a defining set $A$ such that:
$$\forall w \in \Der(E)^{(k)},\ \exists \p \in Aut_A(E),\ \p \cdot (v+w)=\iota_E v. $$
\end{theorem}
So the problem reduces to knowing if our formal inverses constructed in the formal case 
are local or not.  There are two classical examples one due to Poincar\'e and the other one due to Siegel. 
\begin{definition} We say that a  vector $\l \in \CM^n$ lies in the {\em Poincar\'e domain}\index{Poincar\'e domain} if $0$ is not contained in the convex hull of its components.
\end{definition}

The following observation is due to Poincar\'e:
\begin{proposition} If $\l \in \CM^n$ belongs to the Poincar\'e domain then for any $i=1,\dots,n$ the set
$$\{ \left| (\l,I)-\l_i\right|:I \in \NM^n \} \setminus \{ 0 \} $$
is bounded from below and the set
$$\{ I \in \NM^n:(\l,I)-\l_i=0 \}$$
is finite.
\end{proposition}
\begin{proof}
 The condition of being in the Poincar\'e is equivalent to saying that we can rotate the eigenvalues in such way that their real
 part becomes strictly positive.\\
 \vskip0.3cm
 \begin{figure}[htb!]
 \includegraphics[width=0.7\linewidth]{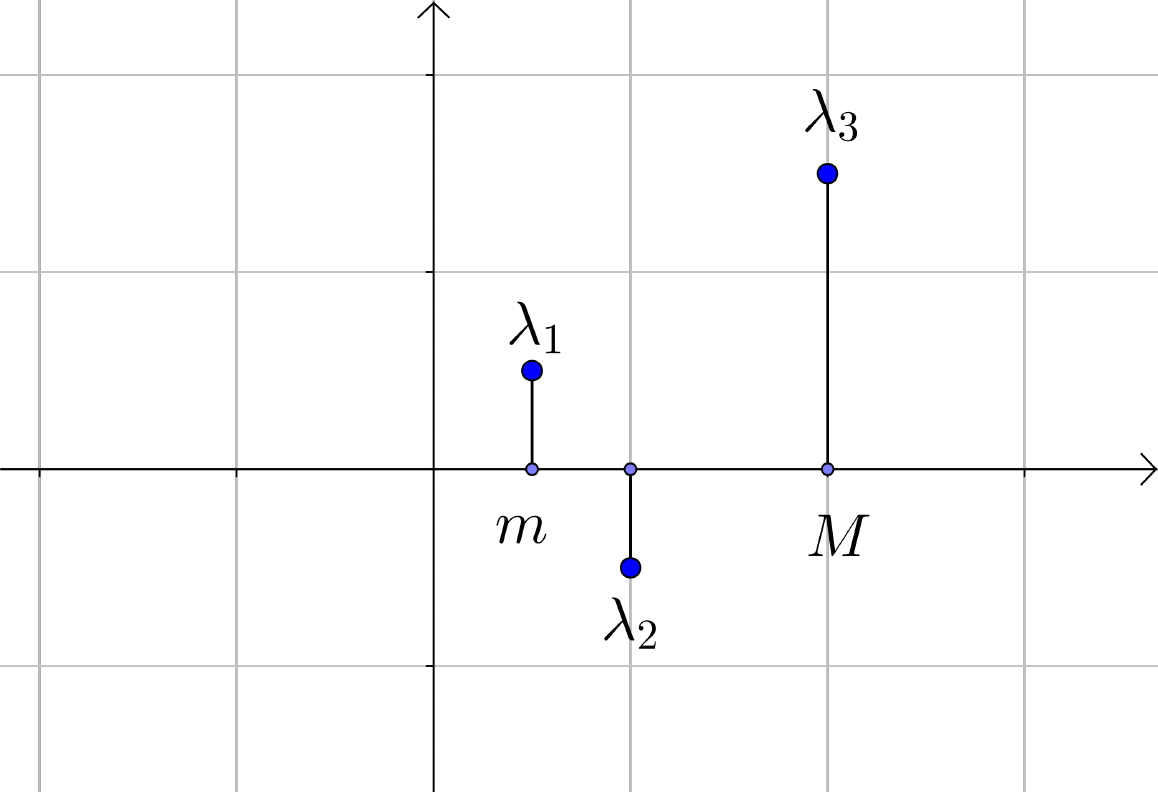}
 \end{figure}
 \vskip0.3cm
 We perform such a rotation denote by $m$ the minimum of these real value and by $M$ its maximum.
 We have
 $$Re\left((\l,I)-\l_i \right) \geq |I|m-M ,$$
 with $|I|=i_1+i_2+\dots+i_n$.
 Therefore if the left hand side vanishes then
 $$|I| \leq \frac{M}{m}. $$
 Thus there are only a finite number of indices $I$ for which this hold.
 
 The semi-group $\G \subset \RM_{\ge 0}$ generated by the real parts of the $\l_i$'s is discrete and therefore the distance from a point $x \notin \G$:
 $$d(x,\G)=\inf\{ |x_g|: g \neq x, g \in \G\} $$
 is strictly positive. This concludes the proof
\end{proof}

\begin{corollary}
Consider a linear vector field
$$v=\sum_{i=1}^n \l_i x \d_{x_i}  .$$
 If the vector $\l=(\l_1,\dots,\l_n)$ belongs to the Poincar\'e domain then the map
$$\Der(R)^{(2)} \to \Der(R)^{(2)},\ w \mapsto [v,w]$$
admits a local right-inverse.
\end{corollary}
\begin{proof}
Put $ R^h=\Ot^h(D)$, the map 
$$j:\Der^{2}(R^h) \to \Der^{2}(R^h),\ x^I \d_i \mapsto \frac{ x^I}{\l_i-(\l,I)}\d_i $$
is local right inverse. Indeed
$$| j (\sum_{I \in \NM^n} a_I x^I \d_i)|_s \, \leq \, | \sum_{I \in \NM^n } \frac{a_I}{\l_i-(\l,I)} x^I \d_i|_s=
\sum_{I \in \NM^n} \frac{|a_I|^2}{|\l_i-(\l,I)|^2} |x^I|_s  .$$
The denominators $|\l_i-(\l,I)|$ are bounded from below by a constant thus the map $j$ is $0$-local in the $\Ot^h$-norm.
As the sheaves $\Ot^h$ and $\Ot^c$ are local equivalent, this proves that $j$ is also $\Ot^c$-local. 
\end{proof}

In this way, we proved the theorem of Poincar\'e:
\begin{theorem} Consider a non-resonant linear vector field
$$v=\sum_{i=1}^n \omega_i x \d_{x_i}.  $$
 If the vector $\l=(\l_1,\dots,\l_n)$ belongs to the Poincar\'e domain, then for any $ w \in \Der(\CM\{ x\})^{(2)}$, there exists  $ \p \in Aut(\CM\{ x \})$ with $ \
 \p(v+w)\p^{-1} =v. $
\end{theorem}
 
\begin{definition} We say that a vector $\l \in \CM^n$ satisfies the {\em Siegel arithmetic condition}\index{Siegel arithmetic condition},
if there exists constants $(C,k)$ such that
$$| \l_i-(\l,I)| \geq \frac{C}{| I |^k} $$
for any $I \in \NM^n$ and any $i=1,\dots,n$.
\end{definition}
The condition is of course similar to the Kolmogorov arithmetic condition and we already observed that it implies locality (Proposition 10.4). Thus we get with the same ease {\em Siegel's theorem:}

\begin{theorem} Consider  a non-resonant linear vector field
$$v=\sum_{i=1}^n \omega_i x \d_{x_i}  $$
and assume that  the vector $\l=(\l_1,\l_2,\dots,\l_n)$ satisfies the Siegel arithmetic condition. Then for any $ w \in \Der(\CM\{ x\})^{(2)}$,
there exists  $ \p \in Aut(\CM\{ x \})$ with $ \ \p(v+w)\p^{-1}=v. $
\end{theorem}

Thus both in the Poincar\'e and in the Siegel cases, the same abstract theorem applies.

\begin{example}
Consider the vector field
$$v=x\d_x-\sqrt{2}y \d_y .$$
The vector $\l=(1,-\sqrt{2})$ does not belong to the Poincar\'e domain. However it satisfies the Siegel arithmetic condition because $\sqrt{2}$ is algebraic (Liouville approximation theorem). Consequently for any $w \in \Der(R)^{(2)}$ there exists an automorphism $\p$ of $R$ such that
$$\p(x\d_x-\sqrt{2}y \d_y+w)\p^{-1}=x\d_x-\sqrt{2}y \d_y. $$
\end{example}

\section{Bibliographical notes}
 
Poincar\'e proved in his thesis that, in the absence of resonances, holomorphic vector fields with linear part in the Poincar\'e domain have holomorphic first integrals:\\

{\sc H. Poincar\'e},  {\em Sur les propri\'et\'es des  fonctions d\'efinies par les \'equations aux diff\'erences 
partielles (Premi\`ere th\`ese, 1879)}, {Gauthiers-Villars}, {Oeuvres de Henri Poincar\'e, Tome I}, {1951}.\\

There were many other works on the subject at the time due to Briot and Bouquet, Picard. Retrospectively, it is clear that Poincar\'e gets much closer to the final Poincar\'e-Dulac theorem than its predecessors. In particular, he introduced what we call now the Poincar\'e domain.  The Poincar\'e-Dulac normal form appears in:\\

{\sc H. Dulac},  {\em Solutions d'un syst\`eme d'\'equations diff\'erentielles dans le voisinage de valeurs singuli\`eres,} Bulletin de la Soci\'et\`e math\'ematique de France {\bf 40} 324-383, (1912).\\
 
Siegel might have been the first to introduce an arithmetic condition in the problem of iteration:\\

{\sc C. L. Siegel}, {\em Iteration of analytic functions}, {Annals of Mathematics}, 43, 607-612, 1942.\\ 

The theorem of Siegel described in this chapter was proved in:\\

{\sc C. L. Siegel,} 	 {\em \"Uber die   Normalform analytischer   Differentialgleichungen in der N\"ahe
einer Gleichgewichtsl\"osung}, Nach. Akad. Wiss. G\"ottingen, math.-phys., 1952.\\

This paper appeared just one year before Kolmogorov's invariant torus theorem. It is amusing to note that 
it is a student of Siegel (Moser) and a student of Kolmogorov (Arnold) who made the next steps towards KAM theory.\\

A classical reference on normal forms of vector field is:\\

 {\sc V.I. Arnold,}  {\em Geometrical methods in the theory of ordinary differential equations}, Grundlehren der Mathematischen Wissenschaften Vol. {\bf 250},  Springer 1988.

\chapter{Invariant varieties in Hamiltonian systems}
In this chapter we return to our original problem of stability of 
quasi-periodic movements and give a complete proof of Kolmogorov's theorem.
In fact, it is a direct application the theory we have developed so far, 
in particular of the normal form theorem in Kolmogorov spaces. As all 
arguments are of a general nature, we prove a generalised Kolmogorov 
theorem with the same ease.
 
\section{The Hamiltonian derivation}
Let us consider a Hamiltonian function of the form
\[ H=\sum \omega_i p_i +\sum a_{ij}p_ip_j+\ldots \in K[[p]]\]
We have seen in Part I, 2.3, that in the formal setting the operator
\[L=\{ H,-\}:K[q,q^{-1}][[p]] \to K[q,q^{-1}][[p]] \]
is diagonal in the monomial basis. In the non-resonant case, the 
kernel is isomorphic to $K[[p]]$, whereas the image of $L$ is the 
space of series:
\[\Im L=\{ \sum_{I \neq 0,J}a_I q^I p^J \}. \]
For $K=\CM$, and $n=1$, any element  $P \in \CM[q,q^{-1}][[p]] $ defines a 
trigonometric polynomial depending on $p$ as a parameter:
$$P(e^{i\theta},e^{-i\theta},p) $$
and the integral
$$\int_{0}^{2\pi}P(e^{i\theta},e^{-i\theta},p) d\theta \in \CM[[p]] $$
is the {\em average} of $P$. If we regard the average as the constant coefficient in the Laurent expansion of $P$, it makes sense in the ring $K[q,q^{-1}][[p]]$ over an arbitrary field $K$ and so the space $\Im L \subset K[[p]]$ consists 
of series with {\em zero average}.
We showed that this fact implies that at the formal level that any deformation of
$H$ is also integrable: all terms depending on $q$'s can be transformed away~(Part I, 2.4).

In chapter I, 2.9,  we also saw that this simple state of affairs does not hold at the analytic level and that the corresponding operator
$$L^{an}=\{ H,-\}:\CM\{q,q^{-1},p\} \to \CM\{q,q^{-1},p\} $$
is of great complexity. Although in the non-resonant case the kernel is
still $\CM\{p\}$, the image is properly contained in the
space of convergent Fourier-series with average zero. 

However if we look at the operator along the torus defined by $p=0$, things simplify: using the identification
\[ \CM\{q,q^{-1}\}=\CM\{q,q^{-1},p \}/(p),\]
the image of the induced operator
\[ L^{an}_0:=\{H,-\}: \CM\{q,q^{-1}\} \to \CM\{q,q^{-1} \}\]
consists exactly of the Fourier series without constant term, provided the frequency $\omega$ satisfies a Diophantine condition.


These properties of the operator $L^{an}_0$ are of crucial importance 
for Kolmogorov's theorem.\\

As we want to identify an appropriate transversal for the group action near $H$,
let us see what happens with $L^{an}_0$ if we add terms to the Hamiltonian.
Adding an element from the ideal $I=(p_1,p_2,\ldots, p_n)$ to $H$ has a drastic
effect: the replacement
$$\sum_{i=1} \omega_i p_i+\dots \mapsto  \sum_{i=1}^n (\omega_i+\a_i) p_i+\dots$$
changes the frequency and therefore might introduce resonances. So this should
be avoided. 
Adding only terms from $I^2$ is a much better idea. For instance, under the 
replacement
$$H=\sum_{i=1} \omega_i p_i+\dots \mapsto  H+\alpha:=\sum_{i=1}^n \omega_i p_i+ \sum_{i,j=1}^n  \a_{ij}p_ip_j +\dots.$$
the induced operator
\[ L^{an}_\a:=\{ H+\alpha,-\}: \CM\{q,q^{-1}\} \to \CM\{q,q^{-1} \}\]
is unchanged:
$$L^{an}_\a=L^{an}_0. $$
This means that vector fields associated to the Hamiltonians $H+\alpha$ 
are all equal along the torus $p=0$.\\  

Recall the idea of a Lagrangian manifold preserved by a flow af a 
Hamiltonian led to the notion of a pair $(H,I)$ in a Poisson 
algebra $A$, 4.8. It consisted of an ideal $I \subset A$ 
with $\{I,I\}\subset I$ and a Hamiltonian $H \in A$ with $\{H,I\} \subset I$. 
The normal space was defined as 
\[N(H,I):=A/\left(\{H,A\}+I^2+H^0(A) \right).\]
Recall also from part I, chapter 4, that the 
first cohomology $H^1(A)$ of a Poisson algebra $A$ can be defined via 
the exact sequence
$$0 \to {\Ht}am(A) \to {\Pt}oiss(A) \to H^1(A) \to 0 $$
as the space of Poisson derivations modulo those that are Hamiltonian.
Furthermore, there is a natural map
\[ H^1(A) \lra N(H,I),\;\;v \mapsto [v(H)] .\]

We first need a convergent version of 4.17.

\begin{proposition}\label{P::NormalAnalytic} 
Let $I$ be the ideal generated by $p_1,\ldots,p_n$ in $A=\CM\{q,q^{-1},p\}$.
Let  $$H=\sum_{i=1} \omega_i p_i+\sum_{ij} a_{ij}p_ip_j+\ldots \in \CM\{p\}$$
be a convergent power series in the $p$-variables. Then:\\

(D): If the vector $\omega=(\omega_1,\dots,\omega_n)$ satisfies 
Kolmogorov's Diophantine condition, then $N(H,I)$ is an $n$-dimensional  
vector space generated by the classes of $p_1,\dots,p_n$.\\
(K): If in addition the matrix $ a_{ij}$ is invertible, then the natural
map $H^1(A) \to N(H,I)$ is surjective.
\end{proposition}

\begin{proof}
(D): As $\{H,I\} \subset I$ we also have $\{H,I^2\} \subset I^2$, so
we obtain a well-defined induced maps
\[ L: A/I^2 \to A/I^2,\;\;a\;\;\; \mod\;\;I^2 \mapsto \{H,a\}\;\; \mod\;\;I^2.\]
There is an exact sequence of the form
\[ 0 \to I/I^2 \to A/I^2 \to A/I \to 0,\]
and  $L$ maps $I/I^2$ to $I/I^2$, so we have also induced maps
$$L_0: A/I \to A/I\;\;\;L_1: I/I^2 \to I/I^2.$$
As we have isomorphisms 
\[A/I=\CM\{q,q^{-1}\},\;\;\;I/I^2 \approx \bigoplus_{i=1}^n A/I p_i\]
both $L_0$ and $L_1$ are determined by the linear part 
$H_0=\sum_{i=1}^n \omega_i p_i$ of the Hamiltonian.
When we split the above sequence in the obvious way
\begin{align*}A/I^2 &\approx A/I \bigoplus I/I^2 \\
 &=A/I \bigoplus \oplus_{i=1}^n A/I p_i ,\end{align*} 
the operator $L$  appears in lower triangular block-form:
\[L=\begin{pmatrix}
      \{ H_0 ,-\} & 0 \\
      \{ a_0,-\} & \{ H_0,- \}
     \end{pmatrix}
\]
where the off-diagonal term is induced by the Poisson-bracket with
$$a_0:=\sum_{i,j=1}^n a_{ij}p_i p_j.$$
The Hamiltonian derivation $\{H_0,-\}$ is identified with the Hadamard 
product with the function
 $$f(q)=\sum_{I \in \ZM^n \setminus \{0\}}(\omega,I) q^I .$$
As $\omega$ is assumed to be Diophantine, there is a well-defined 
Hadamard-inverse function
$$g(q)=\sum_{I \in \ZM^n \setminus \{0\}}(\omega,I)^{-1} q^I .$$
Therefore the operator
 $$L^{-1}=\begin{pmatrix}
      g \star & 0 \\
      -\{ a_0,-\} & g \star
     \end{pmatrix} $$
is an inverse to $L$ over the space of series with zero mean value and
the statement follows.\\
(K): We know from $(D)$ that $N(H,I) $ is of finite dimension,  
generated by the classes of the $p_i$'s. If we now assume that
$$H=\sum_{i=1}^n \omega_i p_i+\sum_{i,j=1}^n a_{ij}p_ip_j+\dots$$
is such that the matrix $(a_{ij})$ is invertible then the map
$$ \CM^{n} \to N(H,I),\;\;\;(a_1,\dots,a_n) \mapsto \sum_{i=1}^n a_i \d_{p_i}(H) $$
has a non-vanishing determinant and is therefore an isomorphism.
\end{proof}

From the above theorem we see that $H$ and $H+\alpha$, $\alpha \in I^2$
all have the ``same'' $n$-dimensional normal space, generated by the 
classes of $p_1,p_2,\ldots,p_n$. Furthermore, as long as the matrix
$a_{ij}$ of the quadratic part has a non-zero determinant, the space
$H^1(A)$ surjects to $N(H,I)$. 
%

Summing up our discussion, we proved that under Kolmogorov conditions
$(D)$ and $(K)$, the map
$$\rho:I^2 \oplus \CM  \oplus {\Pt}oiss(A) \to A,\;\;\;(\a,\b,v) \mapsto v(H)+\a+\b $$
is surjective in the analytic setting.\\

We now want to lift this statement to the level of Kolmogorov spaces. 
We consider the open subset:
$$U \to \RM_{>0} $$
with fibre
 $$U_s=\{ (q,p) \in (\CM^*)^n \times \CM^n : 1-s < | q_i| < 1+s,\ | p_i | < s \} . $$

\begin{proposition} 
\label{P::local_inverse} Consider the Kolmogorov space $A = \Ot^c(U)$ and an analytic function $H \in A$ of the $p$-variables
$$H(p)=\sum_{i=1}^n \omega_i p_i+\sum_{i,j=1}^n a_{ij}p_ip_j+\dots$$
Denote by $I \subset A$ the ideal generated by the $p_i$'s.
Assume that:
\begin{enumerate}
 \item[{\rm (D)}]  the vector $\omega=(\omega_1,\dots,\omega_n)$ satisfies Kolmogorov's Diophantine condition,
 \item[{\rm (K)}]  the matrix $(a_{ij})$ is invertible
\end{enumerate}
 then the natural map
 $$ H^1(A)\to N(H,I), v \mapsto [v(H)]$$
 is surjective and moreover the associated maps
 $$\rho_f:I^2 \oplus \CM \oplus {\Pt}oiss(A)    \to A, (\a,\b,v) \mapsto v(f)+\a+\b, $$
with $f \in H+I^2$, admit uniformly bounded local inverses $j_f$ around $H$.
 \end{proposition}
\begin{proof}
As the mononials define an orthogonal basis of the relative Hilbert space $\Ot^h(U)$,
the orthogonal projections 
 $$\Ot^h(U) \to (I^2)^h(U),\ \Ot^h(U) \to I^h(U)$$ define a Kolmogorov space morphism.
As the sheaves $\Ot^c$ and $\Ot^h$ are local equivalent, these orthogonal projections
induce local maps
$$\Ot^c(U) \to (I^2)^c(U),\ \Ot^h(U) \to I^c(U)$$
In particular, in the direct sum decomposition
$$\Ot^c(U)/\left(I^2(U)\right)^c=\Ot^c(U)/I^c(U) \oplus I^c(U)/\left(I^2(U)\right)^c $$
the orthogonal projections are local\footnote{One may also arrive to the same conclusion using Cartan's theorem $\a$.}. Therefore, the previous formula for $j_f:=L_f^{-1}$:
$$L_f^{-1}=\begin{pmatrix}
      g \star & 0 \\
      -\{ a_f,-\} & g \star
     \end{pmatrix} $$
defines a local operator uniformly bounded as $a_f$ (the quadratic part of $f$) varies around
the quadratic part of $H$. This proves the proposition.
\end{proof}
 
\section{The abstract invariant torus theorem}
Like in singularity theory and for normal forms of vector fields, once the 
formal issues are understood well enough, it becomes easy to lift the 
discussion at the level of Kolmogorov spaces.

A {\em Kolmogorov  algebra $A$}\index{Kolmogorov algebra} is a Kolmogorov space  with a compatible algebra structure: multiplication
and addition should be morphisms of Kolmogorov space.
A Poisson algebra with a Kolmogorov space structure is called a {\em Kolmogorov-Poisson algebra}, if the Poisson derivations are $1$-local, i.e. there is a map
$${\Pt}oiss(A) \to L^1(A,A),H \mapsto \{H,-\}. $$
 
 \begin{theorem}
\label{T::General}Let  $A, B$ be Kolmogorov-Poisson algebras such that $B$ is a flat deformation of $A$ depending on a central parameter $t$ (i.e. multiplication by $t$ is a Casimir operator).\\
Consider a pair $(H,I)$ in $A$ such that the natural map
$$ H^1(A) \to N(H,I)$$ is surjective and that the induced maps
$$\rho_f:  I^2 \oplus H^0(A) \oplus {\Pt}oiss(A) \to A,\ (\a,\b,v) \mapsto v(f)+\a+\b $$
 admit uniformly bounded local right inverses for $f \in H+I^2$. 
Then the space $H+tI^2+H^0(A)$ is a transversal at $H$ in $H+tA$ to the action of  Poisson automorphisms:
$$\forall R \in A,\ \exists \p \in P(A),\ \p(H+tR)=H\ \mod(tI^2 \oplus H^0(A)). $$
\end{theorem}
\begin{proof}
The theorem is a direct consequence of the general normal form theorem 13.5  applied with
 $$E=A,\ M=tA,\ F=tI^2+H^0(A). $$
\end{proof}

From this we immediatly deduce the Kolmogorov invariant torus theorem (Part 1, Chapter 3). We consider the 
cotangent space $T^*(\CM^*)^n$ with usual Darboux coordinates $q_i,p_i$ and take 
its product with a complex line with coordinate $t$. The product space has a natural Poisson  structure over the space of the $t$-variable. The real torus $T$ of $(\CM^*)^n$ embedds inside 
in $T^*(\CM^*)^n \times \CM$ by taking $p_i$'s and $t$ equal to zero.

\begin{theorem} 
Consider the Poisson algebra $A=\CM\{t,q,q^{-1},p\}$ over $B=\CM\{t\}$ and an analytic function $H \in A$ of the
$p$-variables
$$H=\sum_{i=1}^n \omega_i p_i+\sum_{i,j=1}^n a_{ij}p_ip_j+\dots$$
Assume that 
\begin{enumerate}
\item[{\rm (D)}] the vector $\omega=(\omega_1,\dots,\omega_n)$ satisfies Kolmogorov's Diophantine condition
\item[{\rm (K)}] the matrix $(a_{ij})$ is invertible
\end{enumerate}
then for any $R \in A$ there exists a Poisson automorphism   $\p \in {\Pt}oiss(A)$ such that
   $$\p(H+tR)=H\ \mod\ t\,I^2 \oplus \CM\{t\} $$
   with $I=(p_1,\dots,p_n)$. In particular, $H+tR$ admits an invariant ideal isomorphic to $I$.
\end{theorem}
\begin{proof}
Take the product of the open sets
$$U \to \RM_{>0},\ V \to \RM_{>0},  $$
with fibre
\begin{align*}
U_s&=\{ (q,p) \in (\CM^*)^n \times \CM^n : 1-s < | q_i| < 1+s,\ | p_i | < s\}  \\
V_s&= \{ t \in \CM: | t | < s \}
\end{align*}
Define $A^c=\Ot^c(U \times V)$, $B^c=\Ot^c(V)$, $I^c=(p_1,\dots,p_n) \subset A^c$.
By Proposition \ref{P::local_inverse}, the maps
$$\rho_f : (I^2)^c \oplus B^c \oplus {\Pt}oiss(A^c)   \to A^c, (\a,\b,v) \mapsto v(f)+\a+\b $$
 admit uniformly bounded local inverses $j_f$,\ $f \in H+(I^2)^c$. Thus applying the abstract invariant torus theorem~\ref{T::General} with $ A^c = \Ot^c(U \times V)$, $B^c=\Ot^c(V)$, we deduce the Kolmogorov invariant torus theorem.
\end{proof}

\section{The singular torus theorem}
Instead of considering deformation with respect to a parameter $t$, we may consider a Hamiltonian with dominant terms and residual
terms of higher order and we can also mix both approach. We give here a simple example.

The Hamiltonian
\[ \sum_{i=1}^n (p_i^2+\omega_iq_i^2) \]
describing the system of $n$-uncoupled harmonic oscillators, with frequencies $\omega_1,\omega_2,\ldots,\omega_n$.
The $n$ quantities $p_i^2+q_i^2$ Poisson commute and their common zero-set defines the origin in $\R^{2n}$. 

However, in the complex domain it has a more interesting geometry.
A complex symplectomorphism transforms the above Hamiltonian, up to a factor, into the form
\[H=\sum_{i=1}^n \omega_i p_iq_i\]
The elements $p_iq_i$ for $i=1,2,\ldots,n$ Poisson commute and define a Lagrangian variety that 
consists of $2^n$ linear subspaces of dimension $n$:  for a subset $S \subset \{1,2,\ldots,n\}$
we put $p_i=0$ for $i \in S$ and $q_j=0$ for $j \not\in S$. 
 
 We now look what happens if we add a perturbation to the Hamiltonian. Will this Lagrangian variety
persist?

To formulate the question more precisely, let us consider the Poisson algebra
 $$A:=\CM\{ q,p \}:=\CM\{ q_1,\dots,q_n,p_1,\dots,p_n \}$$
of convergent power series, with maximal ideal $\Mt$, endowed with the standard Poisson structure.
Furthermore, we consider the pair $(H,I)$ in $A$ where
\[H=\sum_{i=1}^n \omega_i p_iq_i\]
and 
\[ I:=(p_1q_1,p_2q_2,\ldots,p_nq_n) \subset A\]   
is the involutive ideal generated by the Poisson-commuting generators $p_iq_i$.

   

\begin{theorem} 
Assume that $\omega=(\omega_1,\dots,\omega_n)$ satisfies Kolmogorov's Diophantine condition, 
then for any $R \in \Mt^3$ there exists a symplectic automorphism   $\p \in \Aut(\CM\{ q,p \})$
such that
   $$\p(H+R)=H\ \mod\ I^2 $$
In particular, $H+R$ admits an invariant ideal isomorphic to $I$.
\end{theorem}
\begin{proof}
Consider the open set
$$U \to \RM_{>0},  $$
with fibre
\begin{align*}
U_s&=\{ (q,p) \in \CM^{2n} : | q_i| < s,\ | p_i | < s\} 
\end{align*}
and consider the Kolmogorov-Poisson algebra $E=A=\Ot^c(U)$.
Define 
$$M=E^{(3)}=\left(\Mt^3\right)^c(U),\ I^c=(p_1q_1,\dots,p_nq_n).$$
A straighforward variant of Proposition \ref{P::local_inverse} shows that the maps
$$(I^2)^c \oplus \CM \oplus {\Pt}oiss(A)^{(3)} \to H+M, (\a,\b,v) \mapsto v(f)+\a+\b $$
 admit uniformly bounded local inverse $j_f$, $F  \in H+(I^2)^c $. Thus the theorem is again a direct consequence of the general normal form theorem~13.5.
  
 \end{proof}  
%
 
 \printindex
 \chapter*{Photo Credits}
 \noindent {\em J. Moser} by Konrad Jacobs, Erlangen (1969), front cover, Oberwolfach  Photo Collection, under licence creative commons CC BY-SA 2.0 DE.
{\em K2 in summer} by Bartek Szumski, p. 116; Wikimedia Commons under licence creative commons CC-BY-SA-3.0.\\  

 \appendix
\chapter{Holomorphic functions of several variables}

The theory of analytic spaces and coherent sheaves arose in a long process of merging
function theory with local algebra that can be traced back at least to the lectures of
{ Weierstrass}. The many results obtained by various authors ({ Behnke}, { Stein}, { Oka},...)
were developed to formal perfection by { Cartan} and { Serre} in the  early fifties of 
the last century. There are several excellent text books covering these topics in much detail.
We mention the classical books by { Grauert} and { Remmert} and the more recent three volume series by { Gunning}.
Only the simplest results from this theory are used at several places of the book and we outline here some salient points.

\section{Coherent sheaves}

The theory of sheaves (of abelian groups) is obtained by abstracting from 
the example of sheaves of functions and can be developed in great generality in the 
context of arbitrary topological spaces. One can define homomorphisms of sheaves 
making it into a category, and given a homomorphism of sheaves, one can define the 
kernel, image and cokernel in the category of of sheaves. In this way we obtain for 
any topological space an abelian category $Sh(X)$ of sheaves (of abelian groups) on $X$.
It is fundamental fact that a sequence of sheaves 
\[ 0 \to \Ft \to \Gt \to \Ht \to 0\]
on a topological space $X$ is exact, if and only if for any $x \in X$ the corresponding
sequences on the level of stalks
\[ 0 \to \Ft_x \to \Gt_x \to \Ht_x \to 0\]
are exact sequences of abelian groups.\\

 An analytic space $X$ comes with a sheaf of holomorphic functions $\Ot_X$ on it, 
called the {\em structure sheaf}\index{structure sheaf}.
The prototype of such analytic spaces are $\CM^n$ with its sheaf $\Ot$ described above. 
An open subset  $U \subset \CM^n$  becomes an analytic space on its own by providing it 
with the restriction $\Ot_U$ of the sheaf $\Ot$ to $U$ as its structure sheaf. 
If $f_1,f_2,\ldots,f_r$ are functions holomorphic on $U \subset \CM^n$, then their common 
zero set
\[ X:=\{ x \in U\;\;|\;\;f_i(x)=0,i=1,2,\ldots,r\} \subset U\]
is naturally an analytic space if we define its structure sheaf to be the quotient of $\Ot_U$
by the ideal sheaf $\It=(f_1,f_2,\ldots,f_r) \subset \Ot_U$ generated by the $f_i$:
\[ \Ot_X=\Ot/\It .\]
A general analytic space is locally isomorphic to an example of this kind.\\

From now on we will only consider sheaves of $\Ot$-modules. A {\em free sheaf}\index{free sheaf} 
(of rank $p$) on $X$ is a direct sum of $p$ copies of the structure sheaf:
\[ \Ot^p_X:=\bigoplus_{i=1}^p \Ot_X .\]
Its sections over an open set are just $p$-tuples of holomorphic function on that open set. 

A sheaf $\Ft$ on $X$ is called {\em locally free} (of rank $p$) if each point $x \in X$ has
a neighbourhood $U$ such that $\Ft_U$ is isomorphic to $\Ot_U^p$.
A sheaf $\Ft$ on $X$ is called {\em locally finitely generated} if for any $x \in X$ 
there exists an open neighbourhood $U$ of $x$ and a surjection $\alpha$ 
\[  \Ot_U^p \stackrel{\alpha}{\longrightarrow} \Ft_U \to 0\]
The kernel sheaf of such a homomorphism need not be finitely generated, but if it
is for {\em any} such surjection, the sheaf $\Ft$ is called {\em coherent}\index{coherent sheaf}.

It is a non-trivial theorem that $\Ot$ itself is coherent ({ Oka}'s theorem)\index{Oka's theorem},
but once this is known one may say that a sheaf $\Ft$ on $X$ is coherent precisely 
if each point has a neighbourhood $U$ such that $\Ft_U$ is isomorphic to the {\em cokernel} 
of a map of free sheaves, giving rise to an exact sequence of the form
\[ \Ot_U^p \stackrel{A}{\to} \Ot_U^q \to \Ft_U \to 0,\] 
usually called a {\em presentation}\index{presentation of a sheaf} 
of the sheaf $\Ft$. Here the homomorphism
\[A \in Hom(\Ot_U^p,\Ot_U^q)\]
can be seen as a matrix with entries from $\Ot(U)$. So any coherent sheaf
may by given by such a matrix, but of course, very different matrices may
give rise to the same sheaf $\Ft$.
If $f_1,f_2,\ldots,f_r$ are functions homolorphic on $U$, then we obtain
a map 
\[ \Ot_U^r \stackrel{(f_1,f_2,\ldots f_r)}{\longrightarrow} \Ot_U \]
Its image is the ideal sheaf $\It \subset \Ot_U$ and the cokernel is just
$\Ot_X$, where $X$ is the common zero set of the $f_i$ in $U$. It is
a non-trivial theorem of { Cartan} that $\It$ is a coherent sheaf. 

Given a homomorphism between coherent sheaves, kernel, image and cokernel
are again coherent, and we obtain an abelian category $Coh(U)$ of coherent 
sheaves on $U$.

\section{Cartan's theorem A and B}

Cartan proved that any closed polycylinder $K \subset \CM^n$
has a neighbourhood $U$ containing $K$, such that for any coherent sheaf
the following holds:\\

{\em Theorem A:} The restriction map that associates to a section its germ 
at $a \in U$ 
\[ \Ft(U) \to \Ft_a \]
is surjective. In other words, $\Ft$ is globally generated.\\

{\em Theorem B:} The higher cohomology groups of $\Ft$ vanish:\\ 
\[ H^p(U,\Ft) =0,\;p \ge 1 \;.\]
These cohomology groups $H^p(-)$ can be defined in terms of Cech-complexes.
One may say that 'all' information about the sheaf $\Ft$ on $U$ is contained
in its space of global sections. 

An analytic space $X$ like $U$, for which Theorem A and B hold for all
coherent sheaves, is  called a {\em Stein space}\index{Stein space} and 
this property can be characterised function-theoretically. 
\section{Weierstra{\ss} polynomials}
We denote the ring of convergent power series by
\[ R=\Ot_0=\CM\{x_1,x_2,\ldots,x_n\} .\]
An element $f \in R$ has a power series expansion
\[ f=\sum_{\alpha \in \NM^n} a_\alpha x^\alpha\]
that converges and thus represents a holomorphic function on an open neighbourhood $U$ of $0$. By collecting terms of equal total degree we can write
\[f=f_0+f_1+f_2+\ldots\]
there $f_k$ is a homogeneous polynomial of degree $k$ in the coordinates
$x_1,x_2,\ldots,x_n$. If the constant term $f(0)=f_0$ is non-zero, $f$ is non-zero 
in a  neighbourhood of $0$ and $f$ is a unit in the ring $R$. Otherwise, if $f$ is a non-zero
element of $R$,  the set
\[ V(f):=\{ x \in U \;\;| f(x)=0\}\]
defines a hypersurface which contains the origin.

If $x \in U$ is a further point, then a point $\lambda x$ on the line $\ell$ connecting $0$ and  $x$ lies on $V(f)$
precisely when
\[0=f(\lambda x)=f_1(\lambda x)+f_2(\lambda x)+\ldots=f_0+\lambda f_1(x)+\lambda^2 f_2(x)+\ldots\]

In particular $f_k(x)=0$ for all $k$'s.
The variety of points defined by the vanishing of the polynomials
\[0=f_1(x)=f_2(x)=f_3(x)=\ldots\] 
is conical and the corresponding set in projective space of directions
at $0$ will be called the {\em variety $B$ of bad directions}\index{bad direction} of $f$:
\[B:=\{ P \in \PM^{n-1}\;\;|\;\;f_k(P)=0, k=1,2,3,\ldots\} .\]

Clearly, if $f$ is a non-zero element of $R$, then $B$ is a proper
subset of $\PM^{n-1}$ which consists of the lines through the origin contained in $V(f)$. 
In this case, we have a Zariski dense open complementary set
$\PM^{n-1} \setminus B$ of good directions of $f$ that we simply call 
$f$-directions. 

The smallest number $d$ for which 
\[f_1(x)=f_2(x)=\ldots,f_{d-1}(x)=0,\;\;f_{d}(x) \neq 0\]
is called the {\em multiplicity} of the series $f$.

Assume that $\ell$ be an $f_d$-direction. By a linear change of coordinates, 
we may make $\ell$ into the $x_n$-axis. We write
\[ x=(x',y),\;\; x':=(x_1,x_2,\ldots,x_{n-1}),\;\;y:=x_n\]
and  say that $f$ is {\em $y$-general of order $d$} \index{power series general
with respect to a coordinate}. Such an $f$ has a representation of the form
\[ f= u y^d+a_1 y^{d-1}+a_2y^{d-2}+\ldots+a_{n-1}y+a_n\] 
where $u \in R$ is a unit, and $a_1,a_2,\ldots,a_n \in R',\;\;a_i(0)=0$ and
where we have put
\[R':=\CM\{x_1,x_2,\ldots,x_{n-1}\} .\]

By a {\em Weierstra{\ss} polynomial\index{Weierstra{\ss} polynomial}
of degree $d$} we mean a monic polynomial of degree $d$ in $y$ with coefficients
from $R'$:
\[ g=y^d+a_1 y^{d-1}+\ldots+a_d,\;\;\;a_i \in R', \;\;a_i(0)=0 .\]

\section{The Weierstra{\ss} division theorem}

\begin{theorem} (Weierstra{\ss} division theorem)
If $g$ is a Weierstra{\ss}-polynomial of degree $d$, then for any
$f \in R$ there exist unique $q \in R$ and $r \in R'[y]$ of degree $<d$ in $y$ 
such that
\[ f= q g +r\]
\end{theorem}
\begin{example}
Consider the polynomial
$$ g(x,y)=y^2-x.$$
Then any holomorphic function $f(x,y)$ can be written in the form:
$$f=(y^2-x)g(x,y)+a(x)y+b(x). $$
Thus the theorem implies that $R/(y^2-x)$ is a free $\CM\{ x \}$-module generated by the classes of $1$ and $y$.
\end{example}	

%
%
%

A proof was given in  II.9. section 7. We used an integral
representation, which has the advantage that one obtains estimates that
show that the division and remainder operators are $(d,0)$-local over 
appropriate neighbourhoods.

%

\section{The Weierstra{\ss} preparation theorem}
To do a division by an $y$-general series rather than a Weierstrass polynomial,
one can use the Weierstrass preparation theorem:

\begin{theorem}{\em (Weierstra{\ss} preparation theorem)}\label{T::preparation theorem}
If $f \in R$ is $y$-general of order $d$, there exists a unique Weierstra{\ss} 
polynomial $g$ of degree $d$ and a unit $u \in R$ such that
\[ f=u g\]
\end{theorem}
\begin{proof}
Without loss of generality we may suppose that
$$f(0,y)=y^d+o(y^d) .$$
We choose an neighbourhood $\D \times D$  adapted to $f$.
Fix $x  \in \CM^{n-1}$ and denote by $\g_x$ the oriented circle $\{ x \} \times \d D$ and let
$$y_1(x),\dots,y_d(x) $$
be the solutions of $f(x,y)=0$. Using the residue theorem, we get
$$I_k(x):=\int_{\g_x} \frac{\xi^k\d_y f(x,\xi)}{f(x,\xi)}d\xi=p_k(y_1(x),\dots,y_d(x)), $$
where $p_k$ denotes the $k$-power sums of the variables:
\[ p_k(\a_1,\a_2,\ldots,\a_d):=\sum_{i=1}^d \a_i^d .\]
By the formulas going back to Newton, the elementary symmetric functions
\[ \s_k(\a_1,\a_2,\ldots,\a_d):=\sum_{i_1<i_2<\dots <i_k}\a_{i_1}\a_{i_2}\dots\a_{i_k} \]
can be expressed universally in terms of the $p_k$:
\[\s_1=p_1,\;\;\s_2=\frac{1}{2}(p_1^2-p_2),\ldots,\s_k=N_k(p_1,p_2,\ldots,p_k),\ldots,\]
where the $N_k$ are polynomials. We put
\[J_k(x):=N_k(I_1(x),I_2(x),\ldots,I_k(x)), \;\;\;k=1,2,\ldots, d .\]
The roots of the Weierstra{\ss} polynomial
$$g(x,y)=y^d-J_1(x)y^{d-1}+J_2(x)y^{d-2}+\dots+(-1)^d J_d(x) $$
are now the zeros of $f(x,-)$ and so the function $f/g$ extends holomorphically.
As its value at the origin is $1$, it is a unit $u$ in $R$. Hence $f=u g$. 
\end{proof}
 
 \section{Cartan's theorem $\a$}
We now formulate an important application of the division theorem to the 
case of matrices.
If a vector $v=(v_i) \in R^q$ lies in the image of a matrix map
\[ R^p \stackrel{A}{\longrightarrow} R^q ,\] 
we can write 
\[ v_i= \sum_{j=1}^p A_{ij} u_j\]  
Now in general there are many choices for $u$ and we want to make sure that one
can make a choice where $u$ becomes small when $v$ becomes small. This can be 
achieved by constructing a linear map
$$ R^q \stackrel{B}{\longrightarrow} R^p,$$
with explicit estimates on its coefficients, such that $B$ is a {\em right 
inverse over the image of $A$}:
$$ABA=A. $$
The map $P:=AB$ then satisfies $P^2=ABAB=AB=P$, hence is a projector  
on the sub-space $Im(A)$: $Im(P)=Im(A) \subset R^q$.
As the sequence of $R$-modules
\[ 0 \lra Im(A) \to R^q \to Coker(A) \lra 0\]
usually does not split as $R$-modules, such a map $B$ can not be expected to 
be $R$-linear; the entries of the matrix $B$ will not be holomorphic functions, 
but rather will contain division and remainder operators, but the resulting
singular behavour of $B$ will be under control by the Weierstra{\ss} theorem
with estimate II.9.11.\\

\begin{definition}\index{adapted neighbourhoods}
Let $A \in Mat(p \times q, R)$ be a matrix of germs of holomorphic functions
and let $U$ be an open subset on which the entries belong to $\Ot^c(U)$.
We say a fundamental system of neighbourhoods consisting of polydiscs
$(\D_\rho) \subset U$ is {\em adapted to $A$}, there exists a linear map
$$B=(B_{ij}): R^q   \to   R^p $$
with 
\[ A B A =A ,\]
such that the entries $\| B_{ij}\|_\rho$ are bounded by $C/\| \rho \|^k$ for 
some integer $k>0$ and constant $C$, which depends only on $A$.
\end{definition}

\begin{theorem}[Cartan's Theorem $\alpha$]
 \label{T::Cartan}
Let  $A \in Mat(p \times q, R)$ a matrix with entries from $R$. Then there 
exists a Zariski dense open subset  $G \subset GL_n(\CM)$, such that for 
each $L \in G$ there exist $s_1,s_2,\ldots,s_n >0$, with the property that 
{\em all} polydiscs
\[\Delta(\rho_1,\rho_2,\ldots,\rho_n)\]
with 
\[ 
\begin{array}{rcl}
\rho_1 &\le &s_1\\
\rho_2 &\le &s_2(\rho_1)\\
\rho_3 &\le &s_3(\rho_1,\rho_2)\\
\ldots &\le & \ldots\\
\rho_{n-1} &\le& s_{n-1}(\rho_1,\ldots,\rho_{n-2})\\
\rho_n  &\le& s_n(\rho_1,\rho_2,\ldots,\rho_{n-1})\\
\end{array}
\]
in the coordinates $x'=Lx$ are adapted to $A$.
\end{theorem}

\begin{proof} We follow the arguments given by Cartan. The proof runs
by a double induction on the number $n$ of variables and number $q$ of
rows in the matrix. For $n=0$ there is nothing to proof.
We assume the truth of the theorem for $<n$ variables and an arbitrary 
number of rows. The general theorem then follows if we can show that:

\begin{itemize}
\item[i)]  The theorem holds for the pair $(n,1)$.
\item[ii)] If the theorem holds for $(n,q-1)$, then also for $(n,q)$.
\end{itemize}
 
i) The case $(n,1)$: Consider a linear map
$$A: R^p  \to  R .$$
If it is the zero-map, there is nothing to prove.  
If the map is non-zero, it contains a non-zero function 
$$g=A(u)$$ in its image. Choose a $g$-direction and write $x=(x',y)$. 
By multiplication by a unit we may assume that $g$ is in fact a 
Weierstra{\ss} polynomial:
$$g(x',y)=y^d+a_1(x')y^{d-1}+\dots+a_{d-1}(x')y+a_d(x'), \;\;\;a_i(x') \in R' .$$
Weierstra{\ss}-division of $h \in R$ by $g$ gives a decomposition
\[ h=q(h) g +r(h),\]
where the remainder $r(h)$ belongs to 
$$R'[y]_{<d}=R'\oplus R' y\oplus \ldots \oplus R' y^{d-1} = R^{\prime d}$$
The remainder map $\rho: R \to R^{\prime d}, h \mapsto r(h)$ and the inclusion
$$j: R^{\prime d}=R'[y]_{<d} \to R, (r_0,r_1,\ldots, r_{d-1}) \mapsto \sum_{i=0}^{d-1} r_{d-i} y^i$$ 
back into $R$, satisfy  $j \rho(h)=r(h)$. We consider the $R'$-module
\[ M:=R'[y]_{<d} \cap Im(A),\] 
consisting of all functions in the image of $A$, which are polynomial of 
degree less than $d$ in the $y$-variable. As a sub-module of 
$R'[y]_{<d} =R^{\prime d}$, $M$ is a finitely generated $R'$-module. 
A choice of generators $m_1,\dots,m_k$ for $M$ defines a matrix $A'$
\[ R^{\prime k} \stackrel{A'}{\lra} R^{\prime d},\;\;e_i \to m_i \]
whose image is $M$. Using the induction hypothesis, there exists a matrix 
$B'$
\[ R^{\prime d} \stackrel{B'}{\lra} R^{\prime  k} \]
satisfying $A'B'A'=A'$ and an estimate of the form
$$\| B'_{ij} \| \leq \frac{C'}{\| \rho \|^k} $$
over polydiscs $$\D_{\rho} \subset \CM^{n-1}$$ with $\rho_i \leq s_i$,
$i=1,2,\ldots,n-1$ together with a corresponding Zariski open set 
$\Omega' \subset GL(n-1,\CM)$.

By picking preimages $n_1,\dots,n_k \in R^p$ of the $m_i$ under $A$, 
we obtain a matrix 
$$\alpha: R^{\prime k} \to R^p,\;\;\;e_i \mapsto n_i $$
so that $ j A'=\alpha A$.

Given an element $h \in Im(A)$, we can do Weierstra{\ss}-division by $g$ and write
$$h=q(h)g+r(h) .$$
As $g \in Im(A)$, we have that $r(h) \in M = Im(A) \cap R'[y]_{<d}=Im(A')$, 
so we can write $\rho(h)=A'w$. Now we define

$$Bh:=q(h)u+\alpha B' \rho (h). $$
Clearly, 
$$ 
\begin{array}{rcl}
ABh&=&q(h)Au+A \alpha B' \rho (h)\\
&=&q(h)g+jA'B'A'w\\
&=&q(h)g+jA'w\\
&=&q(h)g+j\rho(h)\\
&=&q(h)g+r(h)=h.\\
\end{array} 
$$
So indeed we have $ABA=A$, and using the Weierstra{\ss} theorem with estimates
II.9.11 for quotient $q(h)$ and remainder
$r(h)$, we get the theorem for $q=1$.\\

ii): We assume now that the theorem holds for the pair $(n,q-1)$. 
Let $\rho: R^q \to R$ denote the projection on the first component, 
$\lambda: R^q \to R^{q-1}$ the projection on the last $q-1$ components,
$j:R^{q-1} \to R^q$ the inclusion backward. Let 
\[a:=\rho A: R^p \to R,\]
and consider $J:R^N \to R^p$ whose image is $Ker(a)$.
We can form the following diagram of maps
\[
\xymatrix{
R^N         \ar[r]^-J     \ar[d]_-{\tilde{A}}       &      R^p    \ar[r]^{a} \ar[d]^-A  & R \ar[d]^{=} \\
R^{q-1}  \ar[r]^j            &   R^q        \ar[r]^{\rho}& R\\
}
\] 
One has
\begin{align*}
 a=\rho A,\\
 AJ=j\tilde{A},\\
\lambda j=Id .\end{align*}
According to our induction hypothesis, we may assume that there exist $b:R \to R^p$ and
$\tilde{B}:R^{q-1} \to R^N$ which satisfy
\begin{align*}aba=a,\\
\tilde{A}\tilde{B}\tilde{A} =\tilde{A}. \end{align*}

In terms of these data, we define a map $B:R^q \to R^p$ by setting
\[ B:=J \tilde{B}\lambda (Id-Ab\rho)+b\rho .\]

{\em Claim:} $ABA=A$.\\

{\em Proof:} Let $v \in R^p$. One has
 $$a(v-bav)=av-abav=av-av=0,$$ 
so we can write 
$v-bav=J(k)$, for some $k \in R^N$. Note that
\[ (Id-Ab\rho)Av=A(Id-b\rho A)v=A(v-ba v)=AJ(k)=j\tilde{A} (k) .\]
So we find
\[
\begin{array}{ccl}
ABAv&=& AJ\tilde{B}\lambda (Id-Ab\rho)Av+Ab\rho Av\\
    &=& j\tilde{A}\tilde{B}\lambda j \tilde{A}(k)+Abav\\
    &=& j\tilde{A}\tilde{B}\tilde{A}(k)+Abav\\
    &=& j\tilde{A}(k)+Abav\\
    &=& AJ(k)+Abav\\
    &=& A(v-bav)+Abav\\
    &=& Av\\
\end{array}
\]
where we only used the above identities between maps. This completes the
proof.

By induction, we have polydiscs adapted to $a$ and $\tilde{A}$ and corresponding estimates
$$   \| b\| \leq \frac{C}{\|\rho\|^k},\;\;\;\| \tilde{B}_{ij} \| \leq \frac{\tilde{C}}{\|{\rho}\|^k} $$ 
for $\rho_i \leq s_i'$, $\rho_i \leq s''_i$ respectively and corresponding 
Zariski open sets $\Omega', \Omega'' \subset GL(n,\CM)$. Put $\Omega=\Omega' \cap \Omega''$ and choose $s_i:=\min(s_i',s_i'')$. The polydiscs $\D_\rho$ are then
adapted to $A$ for $\rho_i \leq s_i$.
 \end{proof}

The theorem can also be formulated in the following way:
\begin{theorem} Consider the ring $R=\Ot_{\CM^d,0}$
 and a map of free $R$-modules:
 $$A:R^m \to R^n. $$
There exists a fundamental system $\D$ of neighbourhoods of the origin, a number $k$ 
and a map $$B:R^n \to R^m $$
such that $B^c \in Hom^{k,0}_{\overline \D}(\left(\Ot^c(\D)\right)^n, \left(\Ot^c(\D)\right)^m)$
defines a  right-inverse  of 
 $$A^c:\left(\Ot^c(\D)\right)^m \to \left(\Ot^c(\D)\right)^n. $$ 
 \end{theorem}

 \section{Polycylinders adapted to a sheaf}
If $U \subset \CM^n$ is an open set with compact closure $K=\overline{U}$
and $A \in Hom(\Ot^p(K),\Ot^q(K))$ a matrix with entries holomorphic on 
a neighbourhood of $K$ then $A$ defines a presentation of a coherent 
sheaf $\Ft$ on $U$:
\[ \Ot^p \stackrel{A}{\longrightarrow} \Ot^q \to \Ft \to 0 .\]
If we now assume that $U$ is a Stein space, then there is an exact 
sequence of sections over $U$: 
\[ \Ot^p(U) \stackrel{A}{\longrightarrow} \Ot^q(U) \to \Ft(U) \to 0 .\]
Furthermore, we obtain a well-defined map of Banach spaces:
\[ \Ot^c(U)^p \stackrel{A}{\longrightarrow} \Ot^c(U)^q.\]
We define $\Ft^c(U)$ to be the cokernel of this map. The resulting
topology obtained on $\Ft^c(U)$ is {\em independent} of the chosen 
presentation.

\begin{theorem} [Theorem of the closed image]
Let $U \subset \CM^n$ be Stein, and let
\[ 0 \to \Ot(U)^p \stackrel{A}{\to} \Ot(U)^q \to \Ft(U) \to 0\]
be a presentation of a coherent sheaf.  Denote by $\Delta \subset U$ an 
$A$-adapted polydisc centered at $a \in U$. Then the image of the map 
\[ \Ot^c(\D)^p \stackrel{A}{\to} \Ot^c(\D)^q \]
is closed and provides the cokernel $\Ft^c(\D)$ with the structure of a 
Banach space.
\end{theorem}
\begin{proof} (Cartan)
An element $g \in \Ot^c(U)^q$ in the closure of the image of the map 
$A:\Ot^c(U)^p \to \Ot^c(U)^q$
can be represented as limit of a sequence of elements $g_k \in \Ot^c(U)^q$ 
in the image of $A$ with 
\[ |g_{k+1}-g_k|_{\Delta} \le \frac{1}{2^k}\]
Using \ref{T::Cartan} we obtain $f_k \in \Ot^c(U)^p$ with  
\[ g_{k+1}-g_k =A f_k\]
and 
\[ |f_k| \le \frac{C}{2^k}\]
Consequently, the series $\sum_{k} f_k$ converges to an element $f$ with
$Af=g$. So the image of $A$ is closed.
\end{proof}

In particular, one can find a fundamental system of neighbourhoods of the 
origin consisting of polydiscs adapted to a coherent sheaf.

\section{The Douady-Pourcin theorem}
For more precise statements one needs to define certain analytic subsets naturally associated to $\Ft$, and their position
with respect to the natural stratification of the boundary $\partial \Delta$ of the polydisc.
Recall that the {\em depth} of finitely generated module $F$ over a local ring $R$ is the length of a maximal $F$-regular sequence. 
For $R=\CM\{x_1,x_2,\ldots,x_n\}$ a module has depth $n$ if and only if 
$F$ free. One defines for $p=0,1,2,\ldots,n$ the sets
\[ S_p(\Ft) :=\{ a \in U\;|\; \textup{depth}_a(\Ft) \le p\} .\]
These are closed analytic subsets
and let $\Delta^{(p)}$ the set of points of $\Delta$ where exactly $p$ of its coordinates of the point that lie on the boundary of the corresponding factor. 
So $\Delta^{(0)}$ is the interior of $\Delta$ and $\Delta^{(n)}$ is what we 
called earlier the edge $e(\Delta)$.\\

Then one has the following beautiful theorem:\\

\begin{theorem} (Pourcin)
The interior of $\Delta$ is $\Ft$-open if and only if for all $p=0,1,2,\ldots,n$ 
\[ S_p(\Ft) \cap \Delta^{(p+1)} =\emptyset .\]
\end{theorem}
So, for $\Delta$ to be privileged, the locus where the depth of $\Ft$ is $ < n$ has to be disjoint from the edge of $\Delta$.  The proof of the theorem is
discussed in the notes below.

\section{Bibliographical notes}

\noindent Classical references on complex analytic geometry and coherent sheaves are:\\

\noindent {\sc H. Grauert, R. Remmert}, {\em Coherent Analytic Sheaves},
Grundlehren der mathematischen Wissenschaften {\bf 265}, Springer-Verlag,
 1984.\\

\noindent According to {\sc Grauert} and {\sc Remmert}, {\sc Weierstra{\ss}} used his 
{\em Vorbereitungssastz} in university courses around $1860$, but the division theorem 
was given only later by:\\

\noindent {\sc L. Stickelberger}, {\em \"Uber einen Satz von Herrn Noether}, 
Math. Ann. {\bf 30}, (1887), 401-409.\\

\noindent Another nice exposition of the theory can be found in the three volumes:\\

\noindent {\sc J. R. Gunning}: {\em Introduction to Holomorphic Functions of Several Variables}: 
I. Function Theory, II. Local theory, III: Homological Theory. (Wadsworth 1990).\\

\noindent Much of the developments on complex analysis around the notion of
coherent sheaves goes back to the classical paper:\\

\noindent {\sc H. Cartan:} 
{\em Id\'eaux de fonctions analytiques de $ n $ variables complexes.}  
Annales scientifiques de l'\'Ecole Normale Sup\'erieure, Vol. 61 (1944), 
p. 149-197.\\

Here one finds a version of the Weierstra{\ss} division theorem with estimates, 
that is used to obtain Theorem $\alpha$ of that paper. One of the consequences
is the statement that ideals in the power series ring are closed in the natural
topology.\\

It seems, the real importance of this theorem was not recognised before Douady used it 
to construct the moduli space of compact analytic subspaces in his thesis:\\

{\sc A. Douady:} {\em Le probl\`eme des modules pour les sous-espaces analytiques compacts d'un espace analytique donn\'e.} Annales de l'institut Fourier (1966). p. 1-95.\\

Douady modified Cartan's notion requiring not only the image to be closed but also to be a direct summand. Douady noticed later that replacing the $C^0$-Banach space structure by the $L^2$-Hilbert one:
\[|f|:=\int_{K}|f(z)|^2. \] 
this additional condition is superfluous. This is discussed in details in:\\ 

{\bf G. Pourcin}, {\em Sous-espaces privil\'egi\'es d'un polycylindre.} 
 Annales de l'Institut Fourier, Vol. 25 (1975), No. 1, pp. 151-193.\\

Using Douady's results on flatness and privilege, Pourcin's found the above quoted theorem relating 
depth of the sheaf to privilege. However, Douady's argument are unclear to us. Therefore it seems extremely desirable to have a more straightforward proof of
this result by elementary means.


 \end{document}